\documentclass[11pt, letterpaper]{amsart}
\usepackage{graphicx, amssymb, color}
\usepackage{amsmath}
\usepackage{graphicx}
\usepackage{enumerate}
\usepackage{dsfont}
\usepackage{subfigure}

\newtheorem{thm}{Theorem}[section]
\newtheorem{cor}[thm]{Corollary}
\newtheorem{lem}[thm]{Lemma}
\newtheorem{prop}[thm]{Proposition}
\theoremstyle{definition}

\newtheorem{ass}[thm]{Assumption}

\theoremstyle{remark}
\newtheorem{rem}[thm]{Remark}

\numberwithin{equation}{section}
\newcommand{\norm}[1]{\left\Vert#1\right\Vert}

\newcommand{\Real}{\mathbb R}

\newcommand{\F}{\mathcal{F}}

\newcommand{\prob}{\mathbb{P}}
\newcommand{\expec}{\mathbb{E}}

\newcommand{\indic}{\mathbb{I}}
\newcommand{\pare}[1]{\left(#1\right)}
\newcommand{\bra}[1]{\left[#1\right]}

\newcommand{\qprob}{\mathbb{Q}}

\newcommand{\nada}[1]{}

\newcommand{\BMO}{\texttt{BMO}}

\newcommand{\C}{\mathcal{C}}
\newcommand{\cS}{\mathcal{S}}
\newcommand{\cM}{\mathcal{M}}
\newcommand{\cY}{\mathcal{Y}}

\newcommand{\bY}{\mathbb{Y}}
\newcommand{\bZ}{\mathbb{Z}}
\newcommand{\cH}{\mathcal{H}}
\newcommand{\fF}{\mathfrak{F}}

\begin{document}

\title[Consumption investment optimization with Epstein-Zin utility]{Consumption investment optimization with Epstein-Zin utility in incomplete markets}\thanks{The author is grateful to Anis Matoussi for inspiring discussions, to Paolo Guasoni for valuable comments on the draft,  to two anonymous referees and the editor Jak\v{s}a Cvitani\'{c} for their precise comments which help me to improve this paper.}

\author[]{Hao Xing}
\address{Department of Statistics,
London School of Economics and Political Science,
10 Houghton st,
London, WC2A 2AE,
UK}
\email{h.xing@lse.ac.uk}

\begin{abstract}
 In a market with stochastic investment opportunities, we study an optimal consumption investment problem for an agent with recursive utility of Epstein-Zin type. Focusing on the empirically relevant specification where both risk aversion and elasticity of intertemporal substitution are in excess of one, we characterize optimal consumption and investment strategies via backward stochastic differential equations. The supperdifferential of indirect utility is also obtained, meeting demands from applications in which Epstein-Zin utilities were used to resolve several asset pricing puzzles. The empirically relevant utility specification introduces difficulties to the optimization problem due to the fact that the Epstein-Zin aggregator is neither Lipschitz nor jointly concave in all its variables.
\end{abstract}

\keywords{Consumption investment optimization, Epstein-Zin utility, Backward stochastic differential equation}

\date{\today}

\maketitle

\section{Introduction}\label{sec: intro}
Risk aversion and elasticity of intertemporal substitution (EIS) are two parameters describing two different aspects of preferences: risk aversion measures agent's attitude toward risk, while EIS regulates agent's willingness to substitute consumption over time. However commonly used time separable utilities force EIS to be the reciprocal of risk aversion, leading to a rich literature on asset pricing anomalies, such as the equity premium puzzle, the risk-free rate puzzle, the excess volatility puzzle, the credit spread puzzle, and etc.

Recursive utilities of \emph{Kreps-Porteus} or \emph{Epstein-Zin} type and their continuous-time analogue disentangle risk aversion and EIS, providing a framework to resolve aforementioned asset pricing puzzles, cf. \cite{Bansal-Yaron} and \cite{Bansal} for the equity premium puzzle and the risk-free rate puzzle, \cite{Benzoni-et-al} for the excess volatility puzzle, and \cite{Bhamra-et-al} for the credit spread puzzle. All these studies require EIS $\psi$ to be larger than $1$ in order to match empirical observations. Bansal and Yaron \cite{Bansal-Yaron} also empirically estimated $\psi$ to be around $1.5$. On the other hand, empirical evidence suggests that risk aversion $\gamma$ is in excess of $1$. It then follows from $\gamma>1$ and $\psi>1$ that $\gamma \psi>1$. Hence an agent with such a utility specification prefers early resolution of uncertainty (cf. \cite{Kreps-Porteus} and \cite{Skiadas}), therefore asks a sizeable risk premium to compensate future uncertainty in the state of economy.

Other than aforementioned utility specification, two other ingredients are also important in these asset pricing applications. First, investment opportunities in these models are driven by some state variables, which usually lead to unbounded market price of risk; for example, Heston model in \cite{Chacko-Viceira}, \cite{Kraft}, and \cite{Liu}, Kim and Omberg model in \cite{Kim-Omberg} and \cite{Wachter-port}. Second, the first step in all these applications is to understand the superdifferential of the indirect utility for the representative agent, because it is the source to read out equilibrium risk-free rate and risk premium, cf. \cite[Appendix]{Bansal-Yaron}. Therefore, it is important to rigorously study the consumption investment problem simultaneously accounting these three ingredients: utility specification, models with unbounded market price of risk, and superdifferential of indirect utility. However, the following literature review shows that, such a study, in a continuous-time setting, was still missing from the literature. This paper fills this gap.

In the seminal paper by Duffie and Epstein \cite{Duffie-Epstein}, stochastic differential utilities (the continuous-time analogue of recursive utilities, cf. \cite{Kraft-Seifried}) are assumed to have Lipschitz continuous aggregators. Hence the Epstein-Zin aggregator, which is non-Lipschitz, is excluded. Schroder and Skiadas \cite{Schroder-Skiadas-JET} studied the case where $\theta = \frac{1-\gamma}{1-1/\psi}$ is positive.\footnote{The parameter $1+\alpha$ in \cite{Schroder-Skiadas-JET} is $\theta$ here. Hence equation (8c) therein implies $\theta>0$.} However the empirically revelent parameter specification $\gamma,\psi>1$ leads to $\theta<0$. Kraft, Seifried, and Steffensen \cite{KSS-FS} studied incomplete market models with unbounded market price of risk, however their assumption on $\gamma$ and $\psi$ (cf. Equation (H) therein) excludes the case $\gamma>1$ and $\psi>1$.

Regarding market models, Schroder and Skiadas \cite{Schroder-Skiadas-JET} studied a complete market with bounded market price of risk. Schroder and Skiadas \cite[Section 5.6]{Schroder-Skiadas-SPA}, Chacko and Viceira \cite{Chacko-Viceira} both considered incomplete markets and Epstein-Zin utility with unit EIS. Chacko and Viceira \cite{Chacko-Viceira}, Kraft, Seifried, and Steffensen \cite{KSS-FS} studied a market model whose investment opportunities are driven by a square root process, leading to unbounded market price of risk.

Regarding the superdifferential of indirect utility, its form can be obtained by a heuristic calculation using the utility gradient approach, cf. \cite{Duffie-Skiadas}. However, rigorous verification needs the aggregator to satisfy a Lipschitz growth condition (cf. \cite{Duffie-Epstein} and \cite{Duffie-Skiadas}), or joint concavity in both consumption and utility variables (cf. \cite{ElKaroui-Peng-Quenez-AAP}). As we shall see later, when $\gamma>1$ and $\psi>1$, the Epstein-Zin aggregator is neither Lipschitz continuous nor joint concave. On the other hand, for Epstein-Zin utility with $\theta>0$, Schroder and Skiadas \cite{Schroder-Skiadas-JET} verified the superdifferential via an integrability condition (cf. \cite[Lemma 2]{Schroder-Skiadas-JET}) and the property that the sum of deflated wealth process and integral of deflated consumption stream is a supermartingale for arbitrary admissible strategy, and is a martingale for the optimal strategy (cf. \cite[Equation (1)]{Schroder-Skiadas-JET}). Both these two conditions are verified in \cite[Theorem 2 and 4]{Schroder-Skiadas-JET} for complete market models with bounded market price of risk.

In this paper, we analyze a consumption investment problem for an agent with Epstein-Zin utility with $\gamma, \psi>1$ and a bequest utility at a finite time horizon. This agent invests in an incomplete market whose investment opportunities are driven by a multi-variate state variable. Rather than the Campbell-Shiller approximation, which is widely applied for utilities with non-unit EIS, we study the exact solution. As illustrated in \cite[Section 6]{KSS-FS}, there can be a sizeable deviation of the Campbell-Shiller approximation from the exact solution, highlighting the importance of exact solution.

A similar problem has also been studied recently by Kraft, Seiferling, and Seifried \cite{Kraft-Seiferling-Seifried}. In this paper, the relation between $\gamma$ and $\psi$ in \cite{KSS-FS} is removed, all configurations of $\gamma$ and $\psi$ are considered including the $\gamma,\psi>1$ case. Verification result is obtained following the utility gradient approach in \cite{Duffie-Skiadas} and \cite{Schroder-Skiadas-JET}, complemented by a recent note of Seiferling and Seifried \cite{Seiferling-Seifried} for the $\gamma, \psi>1$ case. Nevertheless, \cite{Kraft-Seiferling-Seifried} focuses on models with bounded market price of risk (cf. Assumptions (A1) and (A2) therein). This excludes models, such as Heston model and Kim-Omberg models, which are widely used in aforementioned asset pricing applications.  Comparing to \cite{Kraft-Seiferling-Seifried} and all other aforementioned existing results, the current paper extends the previous literature in three respects.

First, in contrast to the utility gradient approach, the verification result is obtained by comparison results for backward stochastic differential equations (BSDE). Rather than employing the dynamic programming method as in \cite{KSS-FS} and \cite{Kraft-Seiferling-Seifried}, optimal consumption and investment strategies are represented by a BSDE solution, cf. Theorem \ref{thm: verification} below. Extending techniques of Hu, Imkeller, and M\"{u}ller \cite{Hu-Imkeller-Muller} and Cheridito and Hu \cite{Cheridito-Hu}, who studied optimal consumption investment problems for time separable utilities, we verify the candidate optimal strategies for Epstein-Zin utility.

Second, our method is designed for market models with unbounded market value of risk. Utilizing Lyapunov functions, borrowed from \cite[Chapter 10]{Stroock-Varadhan}, we prove in Lemma \ref{lem: martingale} below that certain exponential local martingale is martingale, which is a key component of our verification argument.

Third, we verify the superdifferential of indirect utility. Comparing to \cite{Schroder-Skiadas-JET}, the integrability condition in Lemma 2 therein is satisfied when $\gamma,\psi>1$.\footnote{The specification $\gamma,\psi>1$ is related to \cite[Case 3 in page 113]{Schroder-Skiadas-JET}, which established the utility gradient inequality. Even through its proof is independent of market model, it uses the existence and concavity of Epstein-Zin utility, which are established in \cite[Appendix A]{Schroder-Skiadas-JET} under the assumption $\theta>0$. Therefore one needs to replace \cite[Appendix A]{Schroder-Skiadas-JET} by Propositions \ref{prop: EZ bsde} and \ref{prop: EZ concave} below which confirm the existence and concavity of Epstein-Zin utility when $\theta <0$. During the revision of this paper, these properties are also confirmed in \cite{Seiferling-Seifried} for a general semimartingale setting.}  For the second step of verification in \cite{Schroder-Skiadas-JET} and \cite{Kraft-Seiferling-Seifried}, it requires that the sum of deflated wealth process and integral of deflated consumption stream is a supermartingale for any admissible strategy, and is a martingale for the optimal one. We obtain this property (see Theorem \ref{thm: state price} below) as a by-product of our verification result. This result is established for models with unbounded market price of risk, hence meets demands coming from aforementioned applications on asset pricing puzzles.

Our general results in Section \ref{sec: main} are specialized to two examples in Section \ref{sec: application}. There numeric results reveal an interesting phenomenon. As time horizon goes to infinity, convergence of the finite horizon solution to its stationary long run limit is very slow when $\psi>1$. Figure \ref{fig: horizon} shows that this convergence takes at least $60$ years in an empirically revelent utility and market setting. Moreover, the convergence is sensitive to the time discounting parameter: it is much slower when the discounting parameter decreases slightly.  This is in contrast to the $\psi<1$ case, where the convergence is much faster (around $20$ years) and is less sensitive to the time discounting parameter. This observation implies that, in the $\psi>1$ setting, the finite horizon optimal strategy can be far away from its infinite horizon analogue, even when we consider a lifelong consumption investment problem.

The remaining of this paper is organized as follows. After Epstein-Zin utility is introduced in Section \ref{subsec: Epstein-Zin}, the consumption investment problem is introduced and main results are presented in Section \ref{subsec: portfolio opt}. Then main results are specialized in two examples in Section \ref{sec: application}, where general assumptions of main results are verified under explicit parameters restrictions, which include many empirically relevant cases. All proofs are postponed to appendices.

\section{Main results}\label{sec: main}
\subsection{Epstein-Zin preferences}\label{subsec: Epstein-Zin}

We work on a filtered probability space $(\Omega, (\F_t)_{0\leq t\leq T}, \F, \prob)$. Here $(\F_t)_{0\leq t\leq T}$ is the augmented filtration generated by a $k+n-$dimensional Wiener process $B=(W, W^{\bot})$, where $W$ and $W^{\bot}$ are the first $k$ and the last $n$ components, respectively, and satisfies the usual hypotheses of right-continuity and completeness.

Let $\mathcal{C}$ be the class of nonnegative progressively measurable processes on $[0,T]$. For $c\in \C$ and $t<T$, $c_t$ stands for the consumption rate at $t$ and $c_T$ represents a lump sum consumption at $T$. We consider an agent whose preference over $\mathcal{C}-$valued consumption streams is described by a continuous time stochastic differential utility of \emph{Kreps-Porteus} or \emph{Epstein-Zin} type. To describe this preference, let $\delta >0$ represent the discounting rate, $0<\gamma \neq 1$ be the relative risk aversion, and $0<\psi\neq 1$ be the EIS. We focus on the $\gamma>1$ case. In this case,
define the Epstein-Zin \emph{aggregator} $f: [0,\infty) \times (-\infty, 0] \rightarrow \Real$ via
\begin{equation}\label{eq: EZ agg}
 f(c,v) := \delta \,\frac{(1-\gamma)v}{1-\frac{1}{\psi}}
 \bra{\pare{\frac{c}{((1-\gamma)v)^{\frac{1}{1-\gamma}}}}^{1-\frac{1}{\psi}}-1}.
\end{equation}
This is a standard parametrization used, for instance, in \cite{Duffie-Epstein-pricing}.
Given a bequest utility function $U(c) = c^{1-\gamma}/(1-\gamma)$, the \emph{Epstein-Zin utility} over the consumption stream $c\in \mathcal{C}$ on a finite time horizon $T$ is a process $V^c$ which satisfies
\begin{equation}\label{eq: EZ utility}
 V^c_t = \expec_t \bra{\int_t^T f(c_s, V_s^c) \, ds + U(c_T)}, \quad \text{ for all } t\in[0,T],
\end{equation}
where $\expec_t$ stands for $\expec[\cdot |\F_t]$.

\begin{rem}\label{rem: theta=1}
 Epstein-Zin utility generalizes the standard time separable utility with constant relative risk aversion. Indeed, when $\gamma = 1/\psi$, the aggregator reduces to $f(c,v)= \delta \frac{c^{1-\gamma}}{1-\gamma} - \delta v$. Then \eqref{eq: EZ utility} with $t=0$ can be represented explicitly as the standard time separable utility:
 \[
  V_0^c = \expec \bra{\int_0^T \delta e^{-\delta s} \frac{c_s^{1-\gamma}}{1-\gamma} \, ds + e^{-\delta T} U(c_T)}.
 \]
 As discussed in introduction, we are interested in the empirical relevant case where $\gamma>1$ and $\psi >1$. In this case, $\gamma = 1/\psi$ is violated, hence \eqref{eq: EZ utility} is not time separable.
\end{rem}

When $c$ follows a diffusion, the existence of $V^c$ was established by Duffie and Lions \cite{Duffie-Lions} via partial differential equation techniques. We work with a non-Markovian setting and construct $V^c$ via the following BSDE:
\begin{equation}\label{eq: EZ bsde}
 V_t^c = U(c_T) + \int_t^T f(c_s, V^c_s) \, ds - \int_t^T Z^c_s \, dB_s, \quad 0\leq t\leq T.
\end{equation}
Denote
\[
 \theta:= \frac{1-\gamma}{1-\frac{1}{\psi}}.
\]
When $\gamma,\psi>1$, $\theta<0$.
The generator in \eqref{eq: EZ bsde} is
\[
f(c,v)= \delta \frac{c^{1-\frac{1}{\psi}}}{1-\frac{1}{\psi}} \pare{(1-\gamma)v}^{1-\frac{1}{\theta}} -\delta \theta v.
\]
Then $f$ has super-linear growth in $v$ when $\theta<0$. Therefore the BSDE \eqref{eq: EZ bsde} does not have a Lipschitz generator. Nevertheless, consider $(Y_t, Z_t) := e^{-\delta \theta t} (1-\gamma)  (V^c_t, Z^c_t)$ and the following transformed BSDE:
\begin{equation}\label{eq: bsde trans}
 Y_t = e^{-\delta \theta T} c_T^{1-\gamma} + \int_t^T F(s, c_s, Y_s) \, ds - \int_t^T Z_s dB_s, \quad \text{ where } F(t, c_t, y):= \delta \theta e^{-\delta t} c_t^{1-\frac{1}{\psi}} y^{1-\frac{1}{\theta}}.
\end{equation}
When $\theta<0$, the generator $F$ in \eqref{eq: bsde trans} satisfies the \emph{monotonicity condition}, i.e., $y\mapsto F(t, c_t, y)$ is decreasing. This allows us to establish the existence and uniqueness of solutions to \eqref{eq: EZ bsde}, hence define $V^c$ satisfying \eqref{eq: EZ utility}.

Let us introduce the set of \emph{admissible} consumption streams as
\[
\C_a:= \left\{c\in \C\,:\, \expec\bra{\int_0^T e^{-\delta s}c_s^{1-\frac{1}{\psi}}ds} <\infty \text{ and }  \expec\bra{c_T^{1-\gamma}} <\infty\right\}.\footnote{This admissible set is similar to its counterpart in \cite{Cheridito-Hu} for time separable utilities, but is larger than its analogue in \cite{Schroder-Skiadas-JET}, where  $\expec\bra{\int_0^T c_s^\ell ds}<\infty$ for \emph{all} $\ell \in \Real$ is needed for an admissible consumption stream $c$.}
\]

\begin{prop}\label{prop: EZ bsde}
 Suppose $\gamma,\psi>1$ and $c\in \C_a$. Then \eqref{eq: bsde trans} admits a unique solution $(Y, Z)$ in which $Y$ is continuous, strictly positive, and is of class $D$, $\int_0^T |Z_t|^2 dt <\infty$ a.s.. Moreover, $V^c_t:= e^{\delta \theta t} Y_t/(1-\gamma)$, $t\in[0,T]$, satisfies \eqref{eq: EZ utility}.
\end{prop}

\begin{rem}\label{rem: L2 terminal}
 When a BSDE satisfies the monotonicity condition, it is customary to assume its terminal condition to be square integrable, cf. \cite[Theorem 2.2]{Pardoux-note}. However this imposes unnecessary restrictions for later described utility maximization problem, in the sense that the bequest utility needs to be square integrable to define the associated Epstein-Zin utility. Therefore, Proposition \ref{prop: EZ bsde} only asks for the terminal condition to be an integrable random variable.
\end{rem}

Having defined $V^c_0$, we expect that, as a utility functional,  $\C_a \ni c \mapsto V^c_0$ is concave. This would follow from the standard argument when $f(c,v)$ is jointly concave in $c$ and $v$, cf. \cite[Proposition 5]{Duffie-Epstein}. However, calculation shows that $f$ in \eqref{eq: EZ agg} is \emph{not} jointly concave when $\gamma>1$ and $\psi>1$.\footnote{$f$ is jointly concave in $c$ and $v$ if and only if $\gamma\psi\leq1$.} Nevertheless, utilizing an orderly equivalent transformation of $V^c_0$, introduced in \cite[Example 3]{Duffie-Epstein}, the following proposition confirms the concavity of $c\mapsto V^c_0$.

Let us define $(\bY, \bZ):= (Y^{1/\theta}, \frac{1}{\theta} Y^{1/\theta-1}Z)/(1-\frac{1}{\psi})$. Calculation shows that $(\bY, \bZ)$ satisfies
\begin{equation}\label{eq: bsde ord trans}
 \bY_t = e^{-\delta T}\frac{c_T^{1-\frac{1}{\psi}}}{1-\frac{1}{\psi}} + \int_t^T \bra{\delta e^{-\delta s} \frac{c_s^{1-\frac{1}{\psi}}}{1-\frac{1}{\psi}} + \frac12 (\theta-1) \frac{\bZ^2_s}{\bY_s}} \,ds - \int_t^T \bZ_s dB_s.
\end{equation}
Observe that the generator of \eqref{eq: bsde ord trans} is now jointly concave in $(c, \bY, \bZ)$ when $\theta<1$.

\begin{prop}\label{prop: EZ concave}
 When $\gamma, \psi>1$, for any $c, \tilde{c} \in \mathcal{C}_a$, and $\alpha\in[0,1]$, if $\alpha C + (1-\alpha) \tilde{c} \in \mathcal{C}_a$, then
 \[
  \alpha V_0^c + (1-\alpha) V^{\tilde{c}}_0 \leq V^{\alpha c + (1-\alpha) \tilde{c}}_0.
 \]
\end{prop}

\begin{rem}
 The integrability condition in $\mathcal{C}_a$ does not implies the convexity of $\mathcal{C}_a$. Indeed, since $\psi>1$, $\expec[\int_0^T e^{-\delta s} \beta_s^{1-\frac{1}{\psi}}ds]<\infty$ for both $\beta = c$ and $\tilde{c}$ does not imply the same integrability for $\alpha c + (1-\alpha) \tilde{c}$. However Proposition \ref{prop: EZ concave} implies the concavity of $c\mapsto V^c$ on any convex subset of $\mathcal{C}_a$, for example, $\mathcal{C}_a^1 = \{c\in \mathcal{C}_a\,:\, \expec[\int_0^T e^{-\delta s} c_s ds]<\infty\}$.
\end{rem}

\subsection{Consumption investment optimization}\label{subsec: portfolio opt}
Having established the existence of Epstein-Zin utility in the previous section, we consider an optimal consumption investment problem for an agent with such a utility.

Consider a model of a financial market with a risk free asset $S^0$ and risky assets $S=(S^1, \dots, S^n)$ with dynamics
\begin{equation}\label{eq: SDE S}
\begin{split}
 dS^0_t &= S^0_t r(X_t) dt,\\
 dS_t &= \text{diag}(S_t) \bra{\pare{r(X_t) 1_n + \mu(X_t)} dt + \sigma(X_t) dW^\rho_t},
\end{split}
\end{equation}
where $\text{diag}(S)$ is a diagonal matrix with elements of $S$ on the diagonal, $1_n$ is a $n-$dimensional vector with every entry $1$. Given a correlation function $\rho: \Real^k \rightarrow \Real^{n\times k}$ and $\rho^\bot: \Real^k \rightarrow \Real^{n\times n}$, satisfying $\rho \rho' + \rho^\bot (\rho^\bot)' = 1_{n\times n}$ (the $n\times n$ identity matrix), $W^\rho := \int_0^\cdot \rho(X_s) dW_s + \int_0^\cdot \rho^{\bot}(X_s)  dW^\bot_s$ defines a $n-$dimensional Brownian motion. In \eqref{eq: SDE S}, $X$ is a $E-$valued state variable satisfying
\begin{equation}\label{eq: sde X}
 dX_t = b(X_t) dt + a(X_t) dW_t, \quad X_0=x\in E.
\end{equation}
Here $E\subseteq \Real^k$ is an open domain, $r: E \rightarrow \Real$, $\mu: E \rightarrow \Real^n$, $\sigma: E\rightarrow \Real^{n\times n}$, $b:E \rightarrow \Real^k$, and $a: E \rightarrow \Real^{k\times k}$. These model coefficients satisfy following assumptions.

\begin{ass}\label{ass: coeff}
 $r$, $\mu$, $\sigma$, $b$, $a$, and $\rho$ are all locally Lipschitz in $E$; $A:= aa'$ and $\Sigma = \sigma \sigma'$ are positive definite in any compact subdomain of $E$; $r+ \frac{1}{2\gamma} \mu' \Sigma^{-1} \mu$ is bounded from below on $E$, moreover, dynamics of \eqref{eq: sde X} does not hit boundary of $E$ in finite time.
\end{ass}

In the previous assumption, local Lipschitz continuity of coefficients and the nonexplosion assumption combined imply that \eqref{eq: sde X} admits a unique $E$-valued strong solution $X$. When the interest rate $r$ is bounded from below, due to $\tfrac{1}{2\gamma} \mu' \Sigma^{-1} \mu\geq 0$, $r+ \frac{1}{2\gamma} \mu'\Sigma^{-1} \mu$ is bounded from below as well.

An agent, whose preference is described by an Epstein-Zin utility, invests in this financial market. Given an initial wealth $w$, an investment strategy $\pi$, and a consumption rate $c$, the wealth of the agent follows
\begin{equation}\label{eq: wealth}
 d\mathcal{W}^{\pi, c}_t = \mathcal{W}^{\pi, c}_t \bra{(r_t + \pi'_t \mu_t) dt + \pi'_t \sigma_t dW^\rho_t} -c_t dt, \quad \mathcal{W}^{\pi, c}_0=w.
\end{equation}
Throughout the paper, $r_t, \mu_t, \rho_t, \sigma_t$ stand for $r(X_t), \mu(X_t), \rho(X_t)$, and $\sigma(X_t)$, respectively, and the superscript $(\pi, c)$ is sometimes suppressed on $\mathcal{W}$ to simplify notation. A pair of investment strategy and consumption stream $(\pi, c)$ is \emph{admissible} if $c\in \C_a$ and its associated wealth process is nonnegative.
The agent aims to maximize her utility $V^c_0$.

We will further restrict admissible strategies to a permissible set. But let us first characterize the optimal value process via a heuristic argument.
By homothetic property of Epstein-Zin utility, we speculate that utility evaluated at the optimal strategy has the following decomposition\footnote{The decomposition \eqref{eq: value factor} is widely used for (time-separable) power utilities, cf. eg. \cite{Pham}.}:
\begin{equation}\label{eq: value factor}
 V^*_t = \frac{\mathcal{W}_t^{1-\gamma}}{1-\gamma} e^{Y_t}, \quad t\in[0,T],
\end{equation}
where $Y$ satisfies the following BSDE
\begin{equation}\label{eq: op bsde1}
 Y_t = \int_t^T H(s, Y_s, Z_s, Z^\bot_s) \, ds - \int_t^T Z_s \, d W_s - \int_t^T Z^\bot_s dW^\bot_s.
\end{equation}

Let us determine the generator $H$ in what follows. Parameterizing $c$ by $c = \tilde{c} \,\mathcal{W}$, the wealth process satisfies
\[
 \frac{d\mathcal{W}_t}{\mathcal{W}_t} = (r_t  - \tilde{c}_t + \pi'_t \mu_t) dt + \pi'_t\sigma_t dW^\rho_t.
\]
We expect from the standard dynamic programming principle that $\frac{\mathcal{W}_t^{1-\gamma}}{1-\gamma} e^{Y_t} + \int_0^t f\pare{c_s, \frac{\mathcal{W}^{1-\gamma}_s}{1-\gamma} e^{Y_s}} ds$ is a supermartingale martingale for arbitrary strategy, and is a martingale for the optimal strategy. Let us calculate the drift of the previous process. Calculation shows that
\begin{align*}
 d\mathcal{W}^{1-\gamma}_t &= \mathcal{W}_t^{1-\gamma}\bra{(1-\gamma) (r_t - \tilde{c}_t + \pi'_t \mu_t) - \frac{\gamma(1-\gamma)}{2} \pi'_t \Sigma_t \pi_t} dt + (1-\gamma) \mathcal{W}^{1-\gamma}_t \pi'_t \sigma dW^\rho_t.\\
 d e^{Y_t} &= e^{Y_t} \pare{-H(t, Y_t, Z_t, Z^\bot_t) + \frac12 Z_t Z'_t + \frac12 Z^\bot_t (Z^\bot_t)'} dt + e^{Y_t}\pare{Z_t dW_t + Z^\bot_t dW^\bot_t}.
\end{align*}
Therefore, the drift of $\frac{\mathcal{W}_s^{1-\gamma}}{1-\gamma} e^{Y_t} + \int_0^t f\pare{c_s, \frac{\mathcal{W}_s^{1-\gamma}}{1-\gamma} e^{Y_s}} ds$ reads (the time subscript is omitted to simplify notation)
\begin{equation}\label{eq: R drift}
\begin{split}
 \frac{\mathcal{W}^{1-\gamma}}{1-\gamma} e^{Y}&\left\{(1-\gamma) r - \delta \theta + \frac12 Z Z' + \frac12 Z^\bot (Z^\bot)' + \bra{-(1-\gamma) \tilde{c} + \delta \theta e^{-\frac{1}{\theta} Y} \tilde{c}^{1-\frac{1}{\psi}}}\right.\\
 &\left. + \bra{- \frac{\gamma(1-\gamma)}{2} \pi' \Sigma \pi + (1-\gamma) \pi' (\mu + \sigma \rho Z + \sigma \rho^\bot Z^\bot)} - H(\cdot, Y, Z, Z^\bot)\right\}.
\end{split}
\end{equation}
We expect that the drift above is negative for arbitrary $(\pi, \tilde{c})$ and is zero for the optimal strategy. Therefore, the generator $H$ for \eqref{eq: op bsde1} can be obtained by taking supremum on $\pi$ and $\tilde{c}$ in the previous drift and setting it to be zero. Following this direction, we notice that the randomness in $H$ comes only from $X$, which is driven by $W$, moreover, the terminal condition of \eqref{eq: op bsde1} is zero. As a result, $Z^\bot$ is necessarily zero. Therefore, we can reduce \eqref{eq: op bsde1} to
\begin{equation}\label{eq: op bsde}
 Y_t = \int_t^T H(s, Y_s, Z_s) ds -\int_t^T Z_s dW_s,
\end{equation}
where $H$ is given by
\begin{equation}\label{eq: H}
\begin{split}
 H(t, y, z) =& (1-\gamma)r_t  -\delta \theta + \frac12 zz' + \inf_{\tilde{c}} \left[-(1-\gamma) \tilde{c} + \delta \theta e^{-\frac{1}{\theta} y} {\tilde{c}}^{1-\frac{1}{\psi}}\right] \\
 &+ \inf_{\pi} \bra{-\frac{\gamma(1-\gamma)}{2} \pi' \Sigma_t \pi + (1-\gamma) \pi' (\mu_t + \sigma_t \rho_t z')}\\
 =& \frac12 z M_t z' + \frac{1-\gamma}{\gamma} \mu'_t \Sigma^{-1}_t \sigma_t \rho_t z' + \theta \frac{\delta^\psi}{\psi} e^{-\frac{\psi}{\theta} y} + h_t -\delta \theta.
\end{split}
\end{equation}
Here, suppressing the subscript $t$,
\begin{align*}
 &\Sigma := \sigma \sigma'(X),  \quad M := 1_{k\times k}+ \frac{1-\gamma}{\gamma} \rho' \sigma' \Sigma^{-1} \sigma \rho(X), \quad  \text{ and } \quad h:= (1-\gamma) r(X) + \frac{1-\gamma}{2\gamma} \mu' \Sigma^{-1} \mu(X),
\end{align*}
where $1_{k\times k}$ is the $k\times k$-identity matrix.
Recall from Assumption \ref{ass: coeff} that $r + \frac{1}{2\gamma} \mu' \Sigma^{-1} \mu$ is bounded from below. Therefore $\gamma>1$ implies that there exists a positive constant $h_{max}$ such that $h\leq h_{max}$ on $E$.
The infimum in \eqref{eq: H} are due to $\gamma>1$, and they are attained at
\begin{equation}\label{eq: opt st}
 \pi^*_t = \frac{1}{\gamma} \Sigma^{-1}_t \pare{\mu_t + \sigma_t \rho_t Z'_t} \quad \text{ and } \quad \frac{c^*_t}{\mathcal{W}_t^*} = \tilde{c}^*_t = \delta^\psi e^{-\frac{\psi}{\theta} Y_t}, \quad t\in[0,T),
\end{equation}
where $\mathcal{W}^*$ is the wealth process associated to the strategy $(\pi^*, c^*)$. Therefore $\pi^*$ and $c^*$ are candidate optimal strategies.

Coming back to \eqref{eq: op bsde}, even though the generator $H$ has an exponential term in $y$ and a quadratic term in $z$, the parameter specification $\gamma, \psi>1$ allows us to derive a priori bounds on $Y$. In particular, $Y$ is bounded from above by a constant. Meanwhile, since the quadratic term of $z$ in $H$ will be shown to be nonnegative, the lower bound of $Y$ can be obtained by studying a BSDE whose generator does not contain this quadratic term. As a result, a solution to \eqref{eq: op bsde} can be constructed under the following mild integrability conditions.

\begin{ass}\label{ass: op bsde}
 $\,$
 \begin{enumerate}
  \item[i)] $\frac{d\overline{\prob}}{d\prob} = \mathcal{E}\pare{\int \frac{1-\gamma}{\gamma} \mu' \Sigma^{-1} \sigma \rho(X_s) dW_s}_T$ defines a probability measure $\overline{\prob}$ equivalent to $ \prob$;
  \item[ii)] $\expec^{\overline{\prob}}\bra{\int_0^T h(X_s) ds}>-\infty$.
 \end{enumerate}
 Here $\mathcal{E}(\int \alpha_s dW_s)_T:= \exp\pare{-\frac12 \int_0^T |\alpha_s|^2 ds + \int_0^T \alpha_s dW_s}$ denotes the stochastic exponential for $\int_0^T \alpha_s dW_s$.
\end{ass}
\begin{rem}
 Since the generator $H$ contains a linear term in $z$, it is natural to apply Girsanov theorem. Assumption \ref{ass: op bsde} i) allows us to do this and write \eqref{eq: op bsde} under $\overline{\prob}$. This assumption can be checked by explosion criteria; see Section \ref{sec: application} for examples. In ii), the standard exponential moment condition in \cite{Briand-Hu} is avoid, due to the special structure of $H$: the quadratic term in $z$ is nonnegative, and $H(\cdot, 0,0)$ is bounded from above by $h_{max} -\delta \theta$.
\end{rem}

\begin{prop}\label{prop: existence}
 When $\gamma,\psi >1$, let Assumption \ref{ass: op bsde} hold.
 Then \eqref{eq: op bsde} admits a solution $(Y, Z)$ such that, for any $t\in[0,T]$,
        \begin{equation}\label{eq: Y bdd}
         \expec^{\overline{\prob}}_t \bra{\int_t^T h(X_s) \, ds} - \delta \theta(T-t) +\theta \frac{\delta^\psi}{\psi} e^{(\delta \psi-\frac{\psi}{\theta} h_{\max}) T} (T-t) \leq Y_t \leq -\delta \theta (T-t)+\log\expec^{\overline{\prob}}_t\bra{ \exp{\pare{\int_t^T h(X_s)\,ds}}},
        \end{equation}
 and $\expec[\int_0^T |Z_s|^2 ds] <\infty$.
 In particular, since $h\leq h_{max}$, $Y$ is bounded from above by $(h_{max}-\delta \theta) T$.
\end{prop}

Having constructed $(Y, Z)$, the strategies $(\pi^*, c^*)$ in \eqref{eq: opt st} are well defined. To verify their optimality, we need to further restrict the admissible strategies to a permissible set: $(\pi, c)$ is \emph{permissible} if $c\in \C_a$ and $(\mathcal{W}^{\pi,c})^{1-\gamma} e^Y$ is of class $D$ on $[0,T]$.\footnote{When $h$ is bounded from below, for example, both $r$ and $\mu'\Sigma^{-1} \mu$ are bounded, \eqref{eq: Y bdd} implies that $Y$ is bounded from below as well. Then $(\pi, c)$ is permissible if $c\in \mathcal{C}_a$ and $(\mathcal{W}^{\pi, c})^{1-\gamma}$ is of class $D$ on $[0,T]$. This is exactly the definition of permissibility used in \cite{Cheridito-Hu} for the time separable utilities with $\gamma>1$.}

To verify the optimality for $(\pi^*, c^*)$, let us introduce an operator $\fF$. For $\phi\in C^2(E)$,
\begin{equation}\label{eq: fF}
 \fF[\phi]:= \frac12 \sum_{i,j=1}^k A_{ij} \partial^2_{x_i x_j} \phi + \pare{b+ \frac{1-\gamma}{\gamma} a \rho' \sigma' \Sigma^{-1} \mu}' \nabla \phi + \frac12 \nabla \phi' a M a' \nabla \phi + h,
\end{equation}
where the dependence on $x$ is suppressed on both sides. The function $\phi$ in the following assumption is called a Lyapunov function. Its existence facilities proving certain exponential local martingale is in fact martingale, hence verifying optimality of the candidate strategies. This strategy has been applied to portfolio optimization problems for time separable utilities, cf. \cite{Guasoni-Robertson} and \cite{Robertson-Xing-Wishart}.

\begin{ass}\label{ass: phi}
 There exists $\phi\in C^2(E)$ such that
 \begin{enumerate}
  \item[i)] $\lim_{n\rightarrow \infty} \inf_{x\in E\setminus E_n} \phi(x) = \infty$, where $(E_n)_n$ is a sequence of open domains in $E$ satisfying $\cup_{n} E_n = E$, $\overline{E}_n$ compact, and $\overline{E}_n \subset E_{n+1}$, for each $n$;
  \item[ii)] $\fF[\phi]$ is bounded from above on $E$.
 \end{enumerate}
\end{ass}

The final assumption before the main results imposes an integrability assumption on the \emph{market price of risk} $\lambda$. This ensures $\expec\bra{\int_0^T e^{-\delta s} (c_s^*)^{1-1/\psi} ds}<\infty$, hence the admissibility for the candidate optimal consumption stream $c^*$.

\begin{ass}\label{ass: Q^0}
 There exists $\lambda: E \rightarrow \Real^n$  which satisfies $\mu =\sigma \lambda$ and defines a local martingale measure $\qprob^0$ for the discounted asset price via $d\qprob^0/d\prob = \mathcal{E}(\int -\lambda'_s dW^\rho_s)_T$. Moreover
 \begin{equation}\label{eq: Q^0 int}
  \expec^{\qprob^0}\bra{e^{(\psi-1) \int_0^T r_+(X_s) ds} \mathcal{E}\pare{\int \lambda'(X_s)\, dW^0_s}_T^\psi}<\infty,
 \end{equation}
 where $W^0 := W^\rho + \int_0^\cdot \lambda_s ds$ is a $\qprob^0-$Brownian motion and $r_+ = \max\{r, 0\}$.
\end{ass}

\begin{rem}
 The previous assumption is stated under the minimal martingale measure $\mathbb{Q}^0$ (cf. \cite{Follmer-Schweizer}). A careful examination of Lemma \ref{lem: c admissible} shows that this assumption can be replaced by any local martingale measure $\mathbb{Q}$ such that $\expec^{\mathbb{Q}}[\exp((\psi-1) \int_0^T r_+ (X_s) ds) (d\prob/d \mathbb{Q})^\psi]<\infty$.
\end{rem}

\begin{rem}\label{rem: ass phi Q0}
 When $r$ and $\lambda$ are bounded,  Assumption \ref{ass: Q^0} holds automatically and Assumption \ref{ass: phi} is not needed, even for non-Markovian models. Indeed, Assumption \ref{ass: phi} is used to prove the stochastic exponential in Lemma \ref{lem: martingale} below is a martingale. When $r$ and $\lambda$ are bounded, $h$ is bounded, hence $H(\cdot, 0, 0)$ is bounded as well. Therefore, \eqref{eq: Y bdd} implies that $Y$ is bounded, and $\int_0^\cdot Z_s dW_s$ is a \BMO-martingale, cf. eg. \cite[Lemma 3.1]{Morlais}. Then the stochastic exponential in Lemma \ref{lem: martingale} can be proved as a martingale directly. However many models do not have bounded market value of risk. Therefore we retain Assumptions \ref{ass: phi} and \ref{ass: Q^0} in their general forms. These conditions impose some market conditions. In particular, for Markovian models, these conditions will be specified as explicit parameter restrictions in two examples in Section \ref{sec: application} below.
\end{rem}

Now we are ready to state our first main result.

\begin{thm}\label{thm: verification}
 When $\gamma, \psi>1$, let Assumptions \ref{ass: coeff}, \ref{ass: op bsde}, \ref{ass: phi}, and \ref{ass: Q^0} hold. Then $\pi^*$ and $c^*$ in \eqref{eq: opt st} maximize the Epstein-Zin utility among all permissible strategies. Moreover, the optimal Epstein-Zin utility is given by
 \[
  \frac{w^{1-\gamma}}{1-\gamma} e^{Y_0}.
 \]
\end{thm}

The second main result below focuses on the superdifferential of indirect utility. Let us first
define the optimal value process
\begin{equation}\label{eq: op value}
 V^*_t := \frac{(\mathcal{W}^*_t)^{1-\gamma}}{1-\gamma} e^{Y_t}, \quad t\in[0,T],
\end{equation}
where $\mathcal{W}^*$ is the optimal wealth process and $Y$ comes from Proposition \ref{prop: existence}.
Schroder and Skiadas \cite{Schroder-Skiadas-JET} conjectured in Assumption C3 therein that the superdifferential is
\begin{equation}\label{eq: state price}
 D^*_t = w^{\gamma} e^{-Y_0} \exp\bra{\int_0^t \partial_v f(c^*_s, V^*_s) ds} \partial_c f(c^*_t, V^*_t), \quad t\in[0,T].
\end{equation}
The constant $w^\gamma e^{-Y_0}$ in \eqref{eq: state price} normalizes $D^*_0$ to be $1$. Indeed, combining \eqref{eq: EZ agg}, \eqref{eq: opt st} and \eqref{eq: op value}, calculation shows that
\begin{eqnarray}
 D^*_t &=& w^{\gamma} e^{-Y_0} \exp\bra{\int_0^t \delta (\theta-1) ((1-\gamma) V^*_s)^{-\frac1\theta} (c^*_s)^{1-\frac1\psi} ds -\delta \theta t} \delta ((1-\gamma) V^*_t)^{1-\frac1\theta}(c^*_t)^{-\frac1\psi}\nonumber \\
 &=&\exp\bra{\int_0^t (\theta-1) \delta^{\psi} e^{-\frac{\psi}{\theta} Y_s} ds - \delta \theta t} \frac{(\mathcal{W}_t^{\pi^*})^{-\gamma} e^{Y_t}}{w^{-\gamma} e^{Y_0}}
 \label{eq: state price tran}.
\end{eqnarray}
Therefore the previous identity implies that $D^*_0=1$ and $D^*$ is nonnegative.

In \cite[Theorems 2 and 4]{Schroder-Skiadas-JET}, $D^*$ is confirmed to be the superdifferential when the market is complete with bounded market price of risk. This is proved using an integrability assumption in \cite[Lemma 2]{Schroder-Skiadas-JET}, together with the property that $\mathcal{W} D^* + \int_0^\cdot D^*_s c_s ds$ is a supermartingale for arbitrary strategy and is a martingale for the optimal strategy. The integrability assumption in \cite[Lemma 2]{Schroder-Skiadas-JET} is satisfied in our case. Indeed, \eqref{eq: state price tran} shows that $\partial_v f(c^*, V^*) = (\theta -1) \delta^\psi e^{-\frac{\psi}{\theta} Y}-  \delta \theta$, which is bounded due to $\theta<0$ and $Y$ is bounded from above. Now the following result confirms aforementioned property for $\mathcal{W} D^* + \int_0^\cdot D^*_s c_s ds$ in markets with unbounded market price of risk.

\begin{lem}\label{lem: deflator}
 For $D^*$ given by \eqref{eq: state price tran}, it satisfies
 \begin{equation}\label{eq: D sde}
  d D^*_t = - r_t D^*_t dt + D^*_t \pare{-\gamma (\pi^*_t)'\sigma_t dW^\rho_t + Z_t dW_t}, \quad D^*_0=1,
 \end{equation}
 where $Z$ comes from Proposition \ref{prop: existence}.
 Therefore, for any admissible strategy $(\pi, c)$, $\mathcal{W} D^* + \int_0^\cdot D^*_s c_s \, ds$ is a nonnegative local martingale, hence a supermartingale.
\end{lem}

Finally our second main result below confirms that $\mathcal{W}^* D^* + \int_0^\cdot D^*_s c^*_s \, ds$ is in fact a martingale. This result has been proved for recursive utilities with Lipschitz continuous aggregator which is also jointly concave in all its variables, cf. \cite[Theorems 4.2 and 4.3]{ElKaroui-Peng-Quenez-AAP}. However, as we have seen before, none of these conditions are satisfied when $\gamma,\psi>1$.

\begin{thm}\label{thm: state price}
 When $\gamma, \psi>1$, let Assumptions \ref{ass: coeff}, \ref{ass: op bsde}, \ref{ass: phi}, and \ref{ass: Q^0} hold. Then, for optimal strategy $(\pi^*, c^*)$ given in \eqref{eq: opt st}, $\mathcal{W}^* D^* + \int_0^\cdot D^*_s c^*_s \,ds$ is a martingale. Therefore, for any admissible strategy $(\pi, c)$,
 \[
  \expec\bra{\mathcal{W}^{\pi, c}_T D^*_T + \int_0^T D^*_s c_s \, ds} \leq w = \expec\bra{\mathcal{W}^{\pi^*, c^*}_T D^*_T + \int_0^T D^*_s c^*_s \, ds}.
 \]
\end{thm}

In an equilibrium setting where the representative agent has an Epstein-Zin utility, given the consumption stream, equilibrium risk-free rate and risk premium can be read out from $D^*$, providing a framework to study various asset pricing puzzles as discussed in introduction.

\section{Examples}\label{sec: application}
This section specifies general results in the previous section to two extensively studied models, where explicit parameter restrictions are presented so that all assumptions in the previous section are satisfied, hence statements of Theorems \ref{thm: verification} and \ref{thm: state price} hold. These parameter restrictions covers many empirically relevant specifications.

\subsection{Stochastic volatility}\label{subsec: stoch vol}
The following model has a $1-$dimensional state variable, following a square-root process as suggested by Heston, which simultaneously affects the interest rate, the excess return of risky assets and their volatility. This model has been studied by \cite{Chacko-Viceira} for recursive utilities with unit EIS, and \cite{Kraft}, \cite{Liu} for the time separable utilities. This model is specified as follows:
\begin{equation}\label{eq: Heston}
\left\{\begin{array}{l}
 dS_t = \text{diag}(S_t) \bra{\pare{r(X_t) 1_n + \mu(X_t)} dt + \sqrt{X_t} \sigma dW^\rho},\\
 dX_t = b(\ell - X_t) dt + a \sqrt{X_t} dW_t,
\end{array}\right.
\end{equation}
where $r(x)= r_0 + r_1 x$, $\mu(x) = \sigma \lambda x$, with $r_0, r_1\in \Real$, $\sigma\in \Real^{n\times n}, \lambda, \rho \in \Real^n$, and $b, \ell, a\in \Real$. These parameters satisfy
\begin{ass}\label{ass: Heston para}
 $b, \ell, r_1 + \frac{1}{2\gamma} \lambda' \sigma' \Sigma^{-1} \sigma \lambda \geq 0$, $a>0$, and $b\ell >\frac12 a^2$.
\end{ass}
The previous assumption ensures that $X$ takes value in $(0,\infty)$ and $r+ \frac{1}{2\gamma} \mu' \Sigma^{-1} \mu$ is bounded from below, hence Assumption \ref{ass: coeff} is satisfied with $E=(0,\infty)$. The following result provides parameter restrictions such that statements of Theorems \ref{thm: verification} and \ref{thm: state price} hold.

\begin{prop}\label{prop: Heston}
 When $\gamma, \psi>1$, let Assumption \ref{ass: Heston para} and the following parameter restrictions hold:
 \begin{enumerate}
 \item[i)] Either $r_1 >0$ or $\lambda' \sigma' \Sigma^{-1} \sigma \lambda >0$;
 \item[ii)] $(\psi-1) \bra{r_1 + \frac{b \lambda' \rho}{a} + \frac12 \lambda'(\psi 1_{n\times n} - (\psi-1) \rho \rho')\lambda} < \frac{b^2}{2a^2}$.
 \end{enumerate}
 Then statements of Theorems \ref{thm: verification} and \ref{thm: state price} hold.
\end{prop}

In item i), either the interest rate or the excess rate of return has a linear growth component of the state variable. In item ii), the inequality asks either $b$, the mean-reverting speed of the state variable, is large, or the volatility $a$ is small, or $EIS$ is close to $1$.   In particular, when $r_1 =0$ (i.e., constant interest rate) and $\psi>1$, the condition in item ii) is satisfied when
\begin{equation}\label{eq: Heston sufficient}
 b \lambda' \rho \leq -\frac12 \psi a \lambda' \lambda. \footnote{Since $\psi>1$, \eqref{eq: Heston sufficient} yields $\frac{b \lambda' \rho}{a} + \frac12 \psi \lambda'(\psi1_{n\times n} - (\psi-1) \rho \rho') \lambda\leq \frac{b \lambda' \rho}{a}+ \frac12 \psi \lambda' \lambda\leq 0$. Hence the left hand side of the inequality in Proposition \ref{prop: Heston} ii) is negative.}
\end{equation}
This condition covers the empirically relevant specification in \cite{Liu-Pan}, where the parameter values are
\begin{equation}\label{eq: Heston para}
 \lambda=0.47, \quad \sigma =1, \quad b=5, \quad a=0.25, \quad \text{and} \quad \rho =-0.5.
\end{equation}
Taking $\psi=1.5$ from \cite{Bansal-Yaron}, \eqref{eq: Heston sufficient} is verified by calculation.

Figure \ref{fig: Heston} demonstrates the optimal consumption wealth ratio $c^*/\mathcal{W}^*$ and optimal investment fraction $\pi^*$ with respect to volatility $\sqrt{X}$ for different values of risk aversion and EIS. Meanwhile, our numeric results show that EIS has little impact on the optimal investment fraction, and different risk aversions hardly change the optimal consumption wealth ratio. Figure \ref{fig: horizon} compares the optimal consumption wealth ratio for $\psi=0.2$ (top panel) and $\psi=1.5$ (bottom panel). When $\psi=0.2$, the finite horizon optimal consumption wealth ratio converges quickly to its infinite horizon stationary limit. For the parameter specification in \eqref{eq: Heston para}, when the horizon is longer than $20$ years, the time-$0$ optimal consumption strategy is already close to its stationary limit. However, this convergence is much slower when $\psi=1.5$, requiring at least $60$ years when the time discounting parameter $\delta=0.08$. Moreover, in contrast to the $\psi=0.2$ case, the convergence speed is sensitive to $\delta$ when $\psi=1.5$. In this case, the convergence is much slower for smaller value of $\delta$. Intuitively, agent with small discounting parameter is more patient. But she still prefers early consumption when $\psi>1$. Therefore these two competing forces delay the convergence. All comparative statistics is produced by solving the partial differential equation counterpart of \eqref{eq: op bsde} numerically using finite difference methods.

\begin{figure}
\centering
\subfigure{
\includegraphics[scale=0.53]{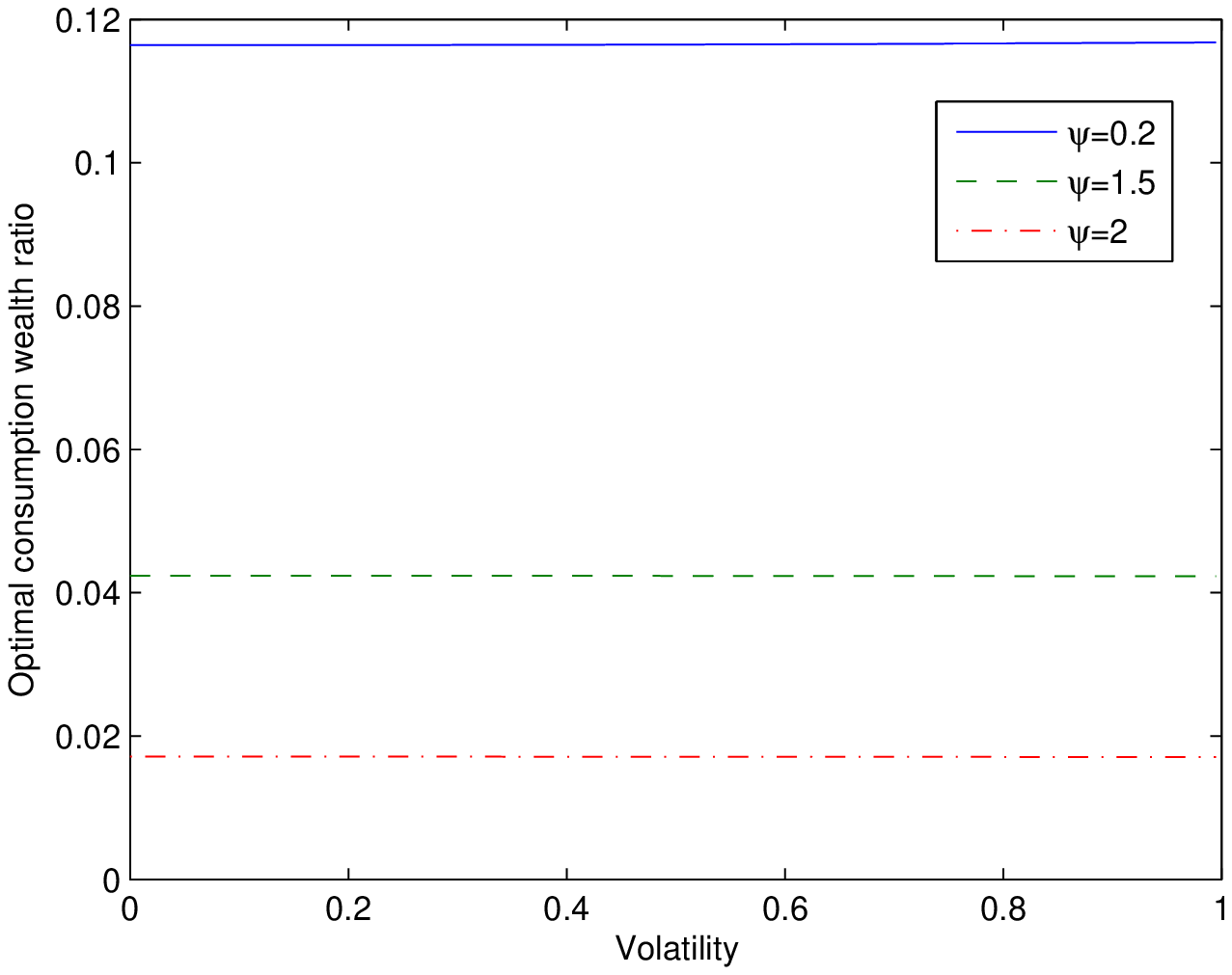}
}
\subfigure{
\includegraphics[scale=0.53]{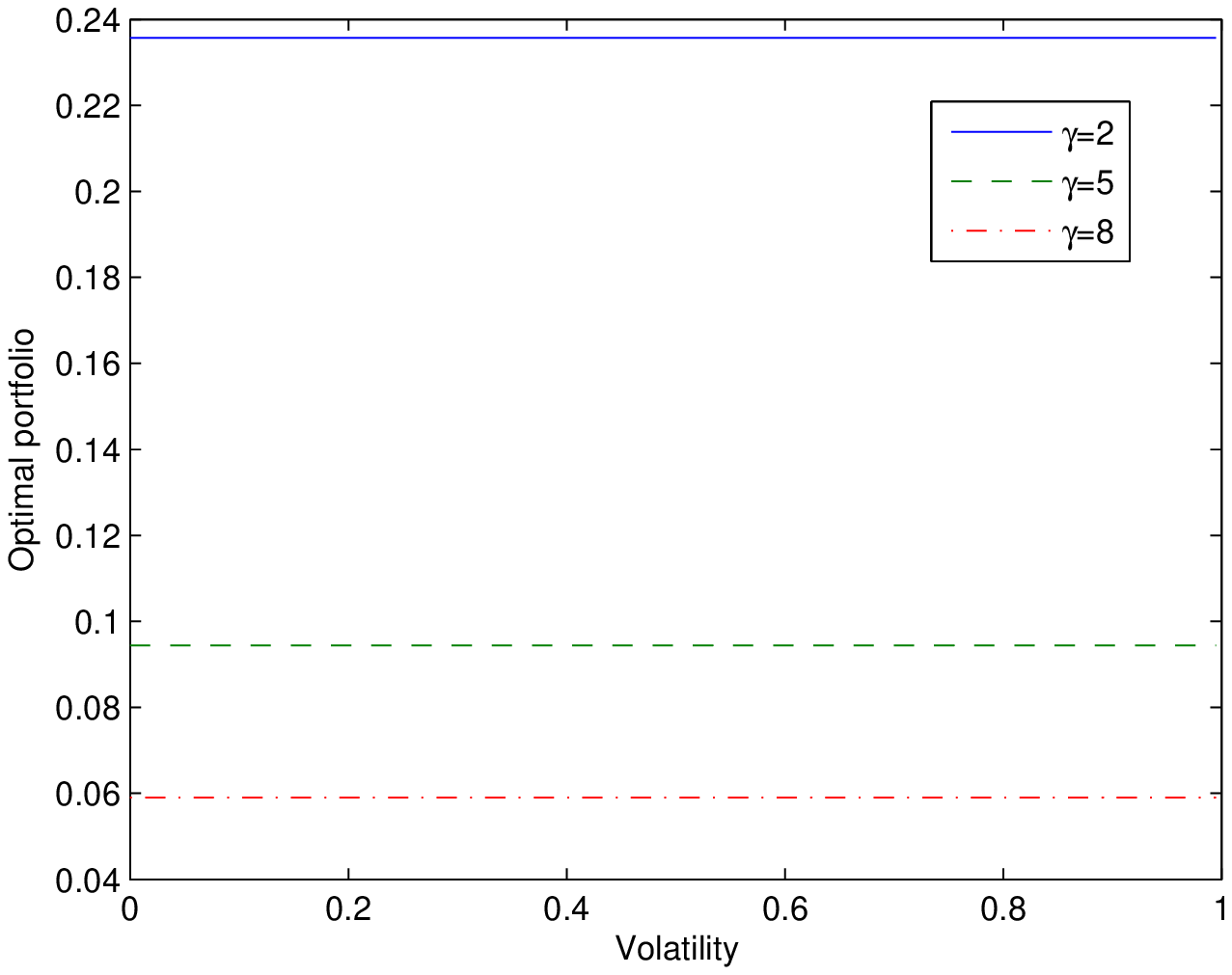}
}
\caption{Both figures use parameters in \eqref{eq: Heston para}, and $r=0.05$, $\delta=0.08$, $\ell=0.0225$. They are both time $0$ values for a problem with time horizon $T=10$ years. The left panel takes $\gamma=5$, and the right panel uses $\psi=1.5$.}
\label{fig: Heston}
\end{figure}

\begin{figure}[h]
\centering
\subfigure{
\includegraphics[scale=0.5]{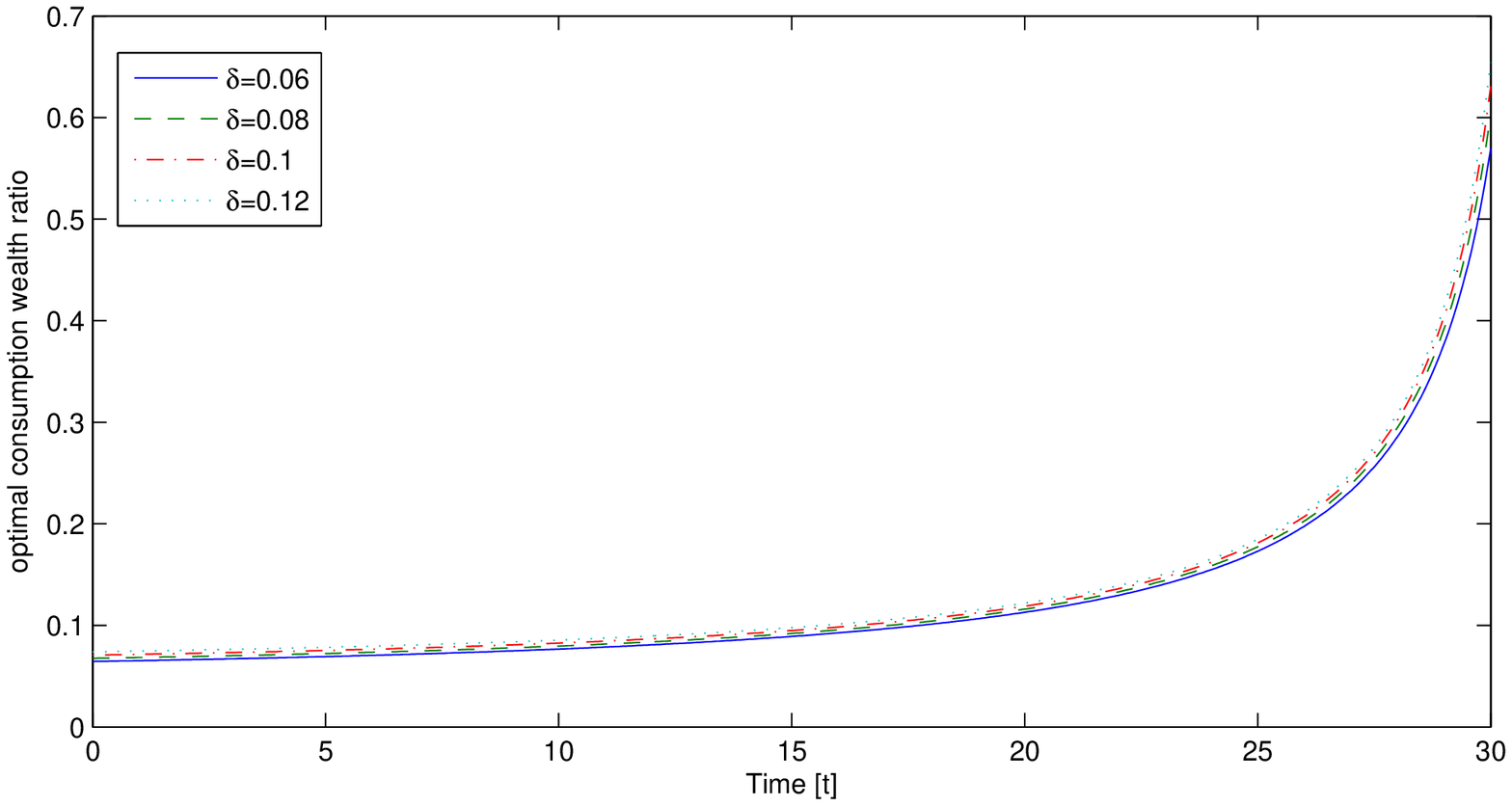}
}
\subfigure{
\includegraphics[scale=0.5]{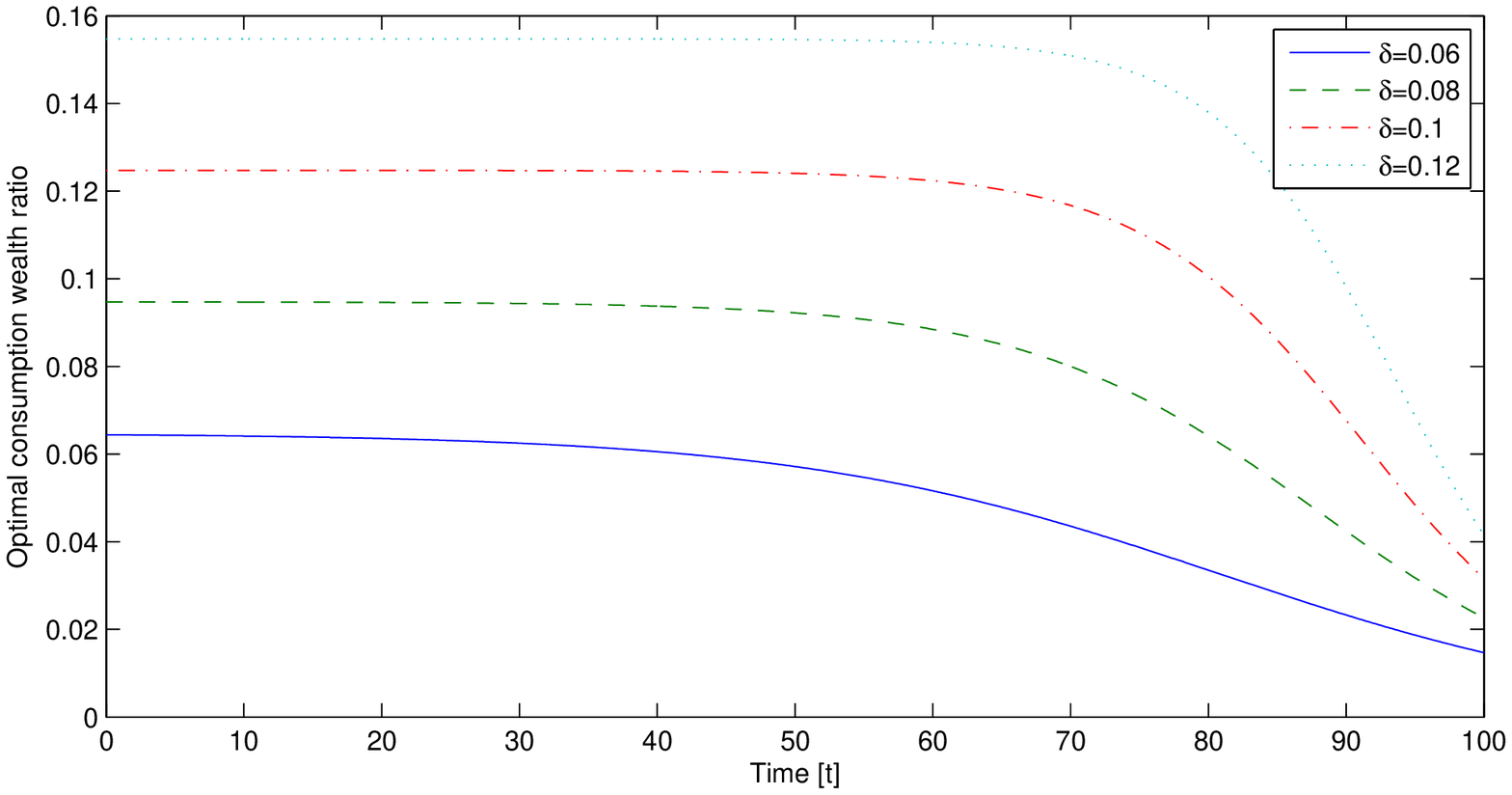}
}
\caption{Optimal consumption wealth ratio as a function of time when volatility is $20\%$. Both figures use parameters in \eqref{eq: Heston para}, $r=0.05$, $\ell=0.0225$, and $\gamma=5$. The upper panel takes $\psi=0.2$ and $T=30$ years. The lower panel fixes $\psi=1.5$ and $T=100$ years.}
\label{fig: horizon}
\end{figure}

\subsection{Linear diffusion}\label{subsec: linear diff}
Both the interest rate and the excess return of risky assets in the following model are linear functions of a state variable, which follows a $1-$dimensional Ornstein-Uhlenbeck process. This model has been studied in \cite{Kim-Omberg} and \cite{Wachter-port} for the time separable utility setting, and in \cite{Campbell-Viceira} for recursive utilities in a discrete time setting. The model dynamics is given by
\begin{equation}\label{eq: OU}
\left\{\begin{array}{l}
 dS_t = \text{diag}(S_t) \bra{\pare{r(X_t) 1_n + \mu(X_t)}} dt + \sigma dW^\rho_t,\\
 dX_t = -bX_t dt + a dW_t,
\end{array}
\right.
\end{equation}
where $r(x)= r_0+ r_1 x$, $\mu(x) = \sigma(\lambda_0+ \lambda_1 x)$, with $r_0, r_1\in \Real, \lambda_0,\lambda_1 \in \Real^n, \sigma\in \Real^{n\times n}, b,a\in \Real$, and $\rho \in \Real^{n}$.  These coefficients satisfy
\begin{ass}\label{ass: OU para}
 $a, b>0$, either $r_1 =0$ or $\lambda_1' \sigma' \Sigma^{-1} \sigma \lambda_1 >0$.
\end{ass}
This assumption implies that Assumption \ref{ass: coeff} is satisfied with $E=\Real$. Under following parameter restrictions, statements of Theorems \ref{thm: verification} and \ref{thm: state price} hold.

\begin{prop}\label{prop: OU}
 When $\gamma, \psi>1$, let Assumption \ref{ass: OU para} and the following parameter restrictions hold:
 \begin{enumerate}
 \item[i)] Either $-b+ \frac{1-\gamma}{\gamma} a\lambda'_1  \sigma' \Sigma^{-1}\sigma \rho<0$ or $\lambda_1' \sigma' \Sigma^{-1}\sigma \lambda_1>0$;
 \item[ii)] $(\psi-1)\bra{\frac{b\lambda'_1 \rho}{a} + \frac12 \lambda'_1 (\psi 1_{n\times n} - (\psi-1)\rho \rho') \lambda_1}< \frac{b^2}{2 a^2}$.
 \end{enumerate}
 Then statements of Theorems \ref{thm: verification} and \ref{thm: state price} hold.
\end{prop}

In the above item i), observe that $(-b+ \frac{1-\gamma}{\gamma} a\lambda'_1  \sigma' \Sigma^{-1}\sigma \rho) X$ is the drift of $X$ under $\overline{\prob}$. Therefore item i) assumes that either $X$ is mean-reverting under $\overline{\prob}$ or the excess rate of return has a linear growth component of the state variable. Item ii) is interpreted similarly as Proposition \ref{prop: Heston} ii) . In particular, when $\psi>1$, the inequality in item ii) is satisfied when
\begin{equation}\label{eq: OU sufficient}
 b \lambda'_1 \rho \leq - \frac12 \psi a \lambda'_1 \lambda_1.\footnote{
 The proof is the same as in footnote 7.}
\end{equation}
This condition already covers many empirically relevant specifications. For example, in \cite{Barberis} and \cite{Wachter-port}, a single risky asset was considered and parameter values (in monthly units) are:
\begin{equation}\label{eq: OU para}
 \lambda_1 = 1, \quad \sigma = 0.0436, \quad b=0.0226, \quad a=0.0189, \quad \rho=-0.935, \quad \text{ and } \quad \psi=1.5.
\end{equation}
Figure \ref{fig: OU} demonstrates the optimal consumption wealth ratio $c^*/\mathcal{W}^{\pi^*}$ and optimal investment fraction $\pi^*$ with respect to the state variable $X$.

\begin{figure}
\centering
\subfigure{
\includegraphics[scale=0.53]{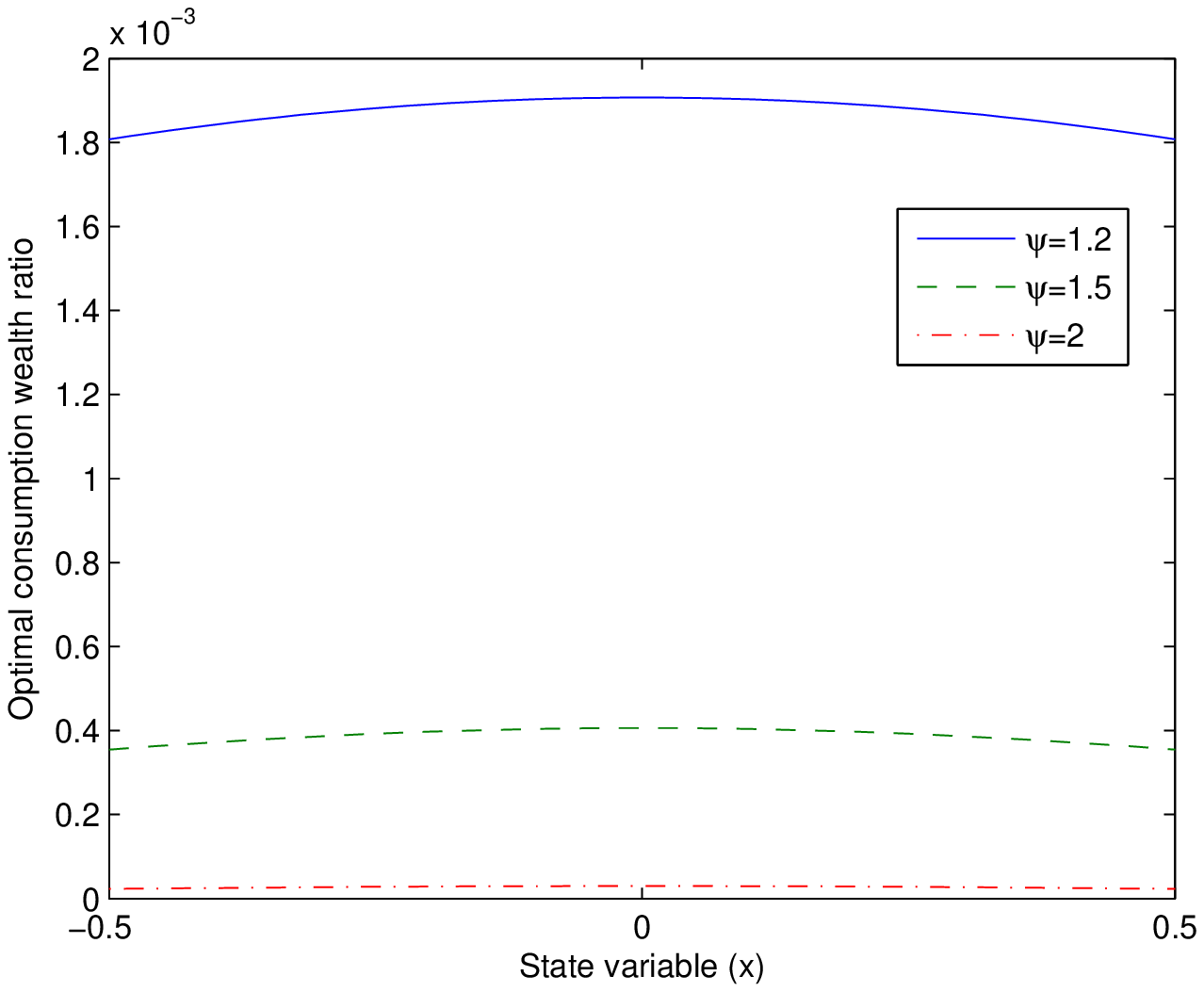}
}
\subfigure{
\includegraphics[scale=0.53]{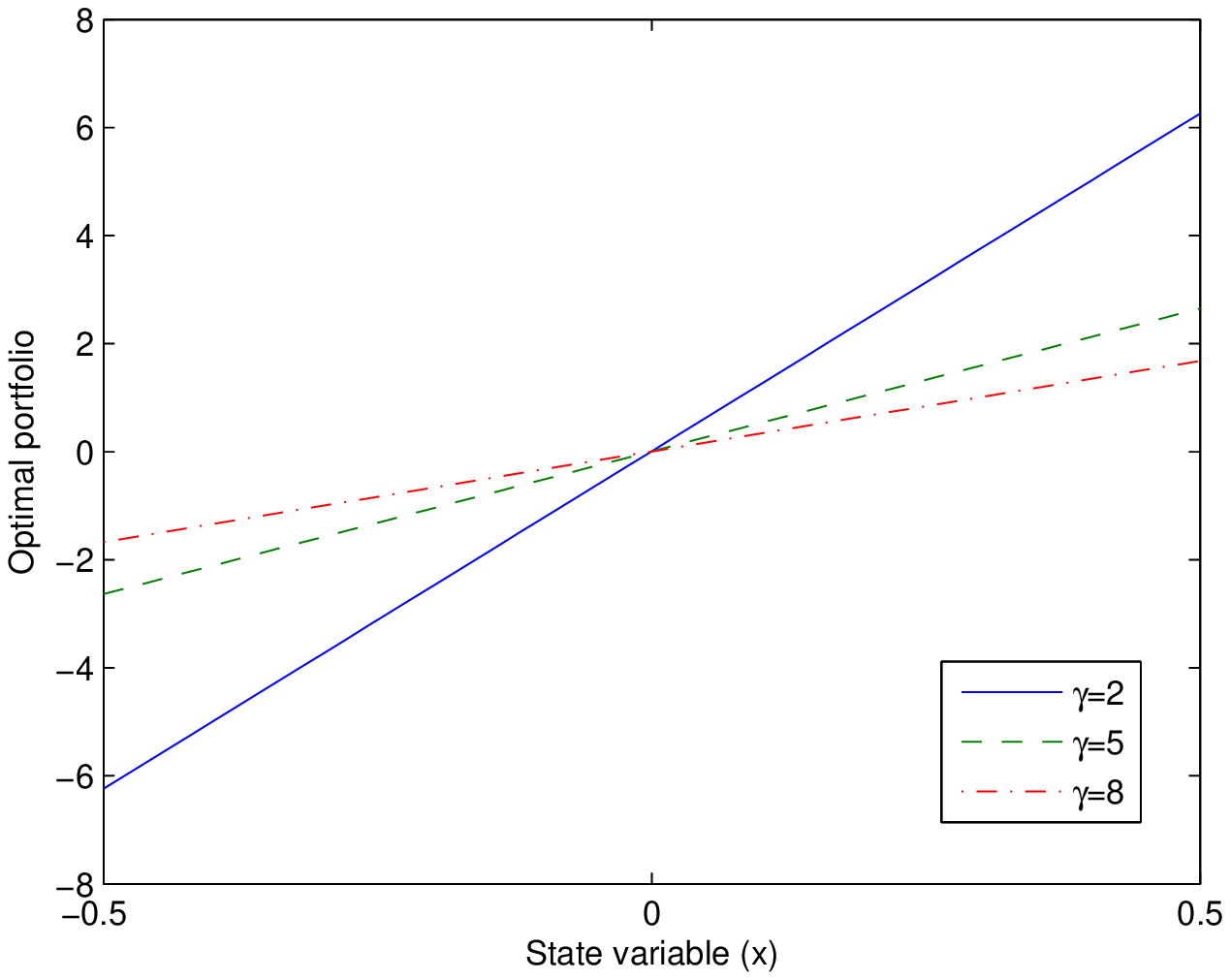}
}
\caption{Both figures use parameters in \eqref{eq: OU para}, and $r=0.0014$, and $\delta=0.0052$. They are both time $0$ values for a problem with time horizon $T=12$ months. The left panel takes $\gamma=5$. The optimal consumption wealth ratio for the $\psi=0.2$ case is much larger than those displaced in the left panel.  The right panel takes $\psi=1.5$.}
\label{fig: OU}
\end{figure}

\appendix
\section{Proofs in Section \ref{subsec: Epstein-Zin}}\label{app: Epstein-Zin}
Let us first introduce several notation which will used throughout the appendices.
\begin{itemize}
 \item Let $\mathcal{S}^2$ denote the space of all $1-$dimensional continuous adapted processes $(Y_t)_{0\leq t\leq T}$ such that the norm $\expec\bra{\sup_{0\leq s\leq T} |Y_s|^2}<\infty$.
 \item Let $\mathcal{S}^\infty$ be the subspace of $\mathcal{S}^2$ such that the norm $\norm{\sup_{0\leq s\leq T} |Y_s|}_\infty<\infty$.
 \item Denote by $\mathcal{T}$ the set of all $\F-$stopping time $\tau$ such that $0\leq \tau\leq T$. The process $Y$ is of class $D$ if the family $\{Y_\tau; \tau\in \mathcal{T}\}$ is uniformly integrable.
 \item Let $\mathcal{M}^2$ denote the class of (multidimensional) predictable processes $(Z_t)_{0\leq t\leq T}$ such that $\expec\bra{\int_0^T |Z_s|^2 ds}<\infty$.
 \item Denote by $\BMO$ the class of martingales $M$ such that $\sup_{\tau \in \mathcal{T}}\norm{\expec[|\langle M\rangle_T - \langle M\rangle_\tau|\,|\, \F_\tau]}_\infty<\infty$.
\end{itemize}

\begin{proof}[Proof of Proposition \ref{prop: EZ bsde}]
 The proof is split into several steps. First when the terminal condition is bounded, the solution is constructed by slightly modifying the proof of \cite[Theorem 2.2]{Pardoux-note}. For general terminal conditions, the solution is obtained by  the localization technique in \cite{Briand-Hu}. Finally, uniqueness is proved and \eqref{eq: EZ utility} is verified. For simplicity of notation, we denote $\xi= e^{-\delta \theta T}c_T^{1-\gamma}$ throughout this proof.

 \vspace{2mm}

 \noindent{\underline{Step 1: Bounded terminal condition.}} When $\xi^2 \leq C$ for some constant $C$, consider the following truncated BSDE:
 \begin{equation}\label{eq: trunc bsde}
  Y^n_t = \xi + \int_t^T F^n(s, c_s, Y^n_s) \, ds - \int_t^T Z^n_s dB_s,
 \end{equation}
 where $F^n(t, c_t, y) = \delta \theta e^{-\delta t} (c_t^{1-\frac{1}{\psi}}\wedge n) (|y| \wedge n)^{1-\frac{1}{\theta}}$. Note that $y\mapsto F^n(t, c_t, y)$ is Lipschitz, in particular, it is differentiable at $y=0$ due to $1-1/\theta >0$. Therefore \eqref{eq: trunc bsde} admits a unique solution $(Y^n, Z^n)\in \cS^2 \times \cM^2$. The first component of such solution is also nonnegative. Indeed, consider \eqref{eq: trunc bsde} with zero as the terminal condition. Such BSDE admits a unique solution $(\tilde{Y}^n, \tilde{Z}^n) \equiv (0,0)$ in $\cS^2 \times \cM^2$. Since $\xi\geq 0$, it follows from the comparison theorem for BSDEs with Lipschitz generators that $Y^n \geq \tilde{Y}^n=0$. On the other hand, since $\theta<0$, $F^n$ is decreasing in $n$, the comparison theorem then implies that $(Y^n)_n$ is decreasing. Hence $Y := \downarrow \lim_{n\rightarrow \infty} Y^n$ is well defined and nonnegative.

 To take the limit of $(Y^n, Z^n)_n$, let us derive the following uniform estimate. Applying It\^{o}'s formula to $(Y^n)^2$ yields
 \begin{equation*}\label{eq: Y Z bdd}
 (Y^n)_t^2 + \expec_t\bra{\int_t^T |Z^n_s|^2 \, ds} = \expec_t\bra{\xi^2} + 2 \, \expec_t\bra{\int_t^T Y^n_s F^n(s, c_s, Y^n_s)\, ds} \leq \expec_t[\xi^2]
  \leq C, \quad \text{for any } t, n,
 \end{equation*}
 where the first inequality follows from $Y^n\geq 0$ and $F^n\leq 0$. The previous estimate yields
 \begin{equation}\label{eq: Y Z bdd 2}
  (Y^n)^2 \leq C \quad \text{and} \quad \expec\bra{\int_0^T |Z^n_s|^2 \, ds} \leq C, \quad \text{ for any } n.
 \end{equation}
 Therefore there exists $Z\in \cM^2$ such that $(Z^n)_n$ converges to $Z$ weakly. Note that $\lim_{n\rightarrow \infty} F^n(t, c_t, y) = F(t, c_t, y)$, $\lim_{n\rightarrow \infty} Y^n = Y$, and
 \begin{align*}
  0\geq F^n(t, c_t, Y^n_t)\geq F(t, c_t, Y^n_t) \geq C^{\frac12 - \frac{1}{2\theta}} \delta \theta e^{-\delta t} c_t^{1-\frac{1}{\psi}}, \quad \text{ for any } n,
 \end{align*}
 where the third inequality holds due to the first estimate in \eqref{eq: Y Z bdd 2}. The dominated convergence theorem then implies that
 \[
  \lim_{n\rightarrow \infty} \int_t^T \left|F^n(s, c_s, Y^n_s) - F(s, c_s, Y_s)\right|ds=0, \quad \text{ for any } t.
 \]
 Now we prove the convergence of $(Z^n)_n$ in $\cM^2$. Applying It\^{o}'s formula to $|Y^n-Y^m|^2$ yields
 \begin{equation}\label{eq: cauchy}
 \begin{split}
  &\expec[|Y^n_0 - Y^m_0|^2] + \expec\bra{\int_0^T |Z^n_s-Z^m_s|^2\, ds} \\ =& 2 \, \expec\bra{\int_0^T \pare{Y^n_s - Y^m_s}\pare{F^n(Y^n_s) - F^m(Y^m_s)}\, ds}\\
  =& 2 \, \expec\bra{\int_0^T \pare{Y^n_s - Y^m_s}\pare{F^n(Y^n_s) - F^n(Y^m_s)}\, ds} + 2 \, \expec\bra{\int_0^T \pare{Y^n_s - Y^m_s}\pare{F^n(Y^m_s) - F^m(Y^m_s)}\, ds}\\
  \leq & 2 \, \expec\bra{\int_0^T \pare{Y^n_s - Y^m_s}\pare{F^n(Y^m_s) - F^m(Y^m_s)}\, ds}\\
  \leq & 4 \delta |\theta| C^{\frac12-\frac{1}{2\theta}} \, \expec\bra{\int_0^T e^{-\delta s}\left|c_s^{1-\frac{1}{\psi}} \wedge n- c_s^{1-\frac{1}{\psi}}\wedge m \right|\, ds},
 \end{split}
 \end{equation}
 where the first inequality holds due to the fact that $y\mapsto F^n(t, c_t, y)$ is decreasing and the second inequality follows from the first estimate in \eqref{eq: Y Z bdd 2}. Since $c\in \C_a$, the dominated convergence theorem implies the right hand side of \eqref{eq: cauchy} converges to zero as $n,m\rightarrow \infty$. Combining the previous convergence with the weak convergence of $(Z^n)_n$, we obtain
 \begin{equation*}\label{eq: Zn conv}
  \lim_{n\rightarrow \infty} \expec\bra{\int_0^T |Z^n_s-Z_s|^2 \, ds} =0,
 \end{equation*}
 The Burkholder-Davis-Gundy inequality then implies
 \[
  \prob-\lim_{n\rightarrow \infty} \sup_{t\leq T} \left|\int_t^T (Z^n_s - Z_s) dB_s\right| =0,
 \]
 where $\prob-\lim$ stands for the convergence in probability. Passing to a subsequence, we obtain almost sure convergence. Therefore, sending $n\rightarrow \infty$ in \eqref{eq: trunc bsde}, we obtain that $(Y, Z)\in \cS^\infty \times \cM^2$ solves \eqref{eq: bsde trans} and $Y$ is nonnegative. Moreover, since
 \[
  |Y^n_t - Y^m_t| \leq \int_t^T \left|F^n(s, c_s, Y^n_s) - F^m(s, c_s, Y^m_s)\right| ds + \left|\int_t^T (Z^n_s - Z^m_s) dB_s\right|,
 \]
 after taking limits on $m$ and supremum over $t$, we obtain
 \[
  \sup_{t\leq T} |Y^n_t - Y_t| \leq \int_0^T \left|F^n(s, c_s, Y^n_s) - F(s, c_s, Y_s)\right| ds + \sup_{t\leq T}\left|\int_t^T (Z^n_s - Z_s) dB_s\right|.
 \]
 Therefore $(Y^n)_n$ converges to $Y$ uniformly in $t$, implying that $Y$ is a continuous process.

 \vspace{2mm}

 \noindent{\underline{Step 2: General terminal condition.}} When $\xi$ is not bounded, set $\xi^n := \xi \wedge n$ and consider
 \[
  Y^n_t = \xi^n + \int_t^T F(s, c_s, Y^n_s) \, ds - \int_t^T Z^n_s \, dB_s.
 \]
 Results from the previous step imply that this BSDE admits a solution $(Y^n, Z^n)\in \cS^\infty \times \cM^2$ with $Y^n \geq 0$. Moreover, since $F\leq 0$, $Y^n_t \leq \expec_t[\xi]$ for all $n$ and $t\in[0,T]$. This a priori bound allows us to construct a solution to \eqref{eq: bsde trans} via the localization technique in \cite{Briand-Hu}. We outline the construction below.

 Consider $\tau_k:= \inf\{t\geq 0\,:\, \expec_t[\xi]\geq k\}\wedge T$ for each $k\in \mathbb{N}$.
 Then $(Y^{n,k}_t, Z^{n,k}_t) := (Y^n_{t \wedge \tau_k}, Z^n_t \indic_{\{t\leq \tau_k\}})$ satisfies the following BSDE
 \[
  Y^{n,k}_t = Y^n_{\tau_k} + \int_t^T \indic_{\{s\leq \tau_k\}} F(s, c_s, Y^{n,k}_s)\, ds -\int_t^T Z^{n,k}_s dB_s.
 \]
 Since $0\leq Y^{n,k}_s \leq \expec_{s\wedge \tau_k}[\xi] \leq k$, we have
 \[
  0\geq F(s, c_s, Y^{n,k}_s) \geq \delta \theta k^{1-\frac{1}{\theta}} e^{-\delta s} c_s^{1-\frac{1}{\psi}}.
 \]
 Then $c\in \C_a$ implies $\expec[\int_0^T F(s, c_s, Y^{n,k}_s) ds]<\infty$. On the other hand, since $\xi^n \leq \xi^{n+1}$ and $y\mapsto F(\cdot, \cdot, y)$ satisfies the monotonicity condition, then the comparison result (cf. \cite[Theorem 2.4]{Pardoux-note}) implies $Y^{n,k} \leq Y^{n+1,k}$. Utilizing the same argument as in Step 1, we obtain $\tilde{Y}^k := \uparrow\lim_{n} Y^{n,k}$ and $\tilde{Z}^k\in \mathcal{M}^2$ such that $\lim_{n} Z^{n,k} =\tilde{Z}^k$ in $\mathcal{M}^2$, and $(\tilde{Y}^k, \tilde{Z}^k)$ solves the BSDE
 \begin{equation}\label{eq: BSDE tY}
  \tilde{Y}^k_t = \tilde{Y}^k_{\tau_k} + \int_t^T \indic_{\{s\leq \tau_k\}} F(s, c_s, \tilde{Y}^k_s) \, ds -\int_t^T \tilde{Z}^k_s dB_s,
 \end{equation}
 where $\tilde{Y}^k_{\tau_k} = \uparrow \lim_n Y^{n}_{\tau_k}$. Following from the definition of $(\tilde{Y}^k, \tilde{Z}^k)$, $\tilde{Y}^{k+1}_{t\wedge \tau_k} = \tilde{Y}^k_t$ and $\tilde{Z}^{k+1}_t \indic_{\{t\leq \tau_k\}} = \tilde{Z}^k_t$. Therefore we define
 \[
  Y_t := \tilde{Y}^k_t\quad \text{and} \quad Z_t := \tilde{Z}^k_t, \quad \text{when } t\in [0, \tau_k].
 \]
 This construction implies $\lim_{t\rightarrow T}Y_t= \xi$. Indeed, on $\{\xi\leq k\}$, $\tau_k=T$ and $\lim_{t\rightarrow T}Y^n_{t} =\xi$ for any $n\geq k$. Therefore $\lim_{t\rightarrow T}Y_{t} = \lim_{t\rightarrow \tau_k} \tilde{Y}^k_{t} = \lim_{t\rightarrow \tau_k} Y^{n,k}_{t}= \lim_{t\rightarrow T} Y^n_t=\xi$ on $\{\xi\leq k\}$ when $n\geq k$. This implies $\lim_{t\rightarrow T}Y_t=\xi$, since $\uparrow\lim_{k\rightarrow \infty} \{\xi \leq k\}=\Omega$. Now sending $k\rightarrow \infty$ on both sides of \eqref{eq: BSDE tY}, we confirm that $(Y, Z)$ solves \eqref{eq: bsde trans}. By this construction, $Y$ is continuous and satisfies $0\leq Y_t \leq \expec_t[\xi]$ for $t\in [0,T]$, hence $Y$ is of class $D$. The same argument as in \cite[Page 612]{Briand-Hu} shows $\int_0^T |Z_t|^2 dt <\infty$.

 \vspace{2mm}

 \noindent{\underline{Step 3: Remaining statements.}} For future reference, we prove a comparison result for \eqref{eq: bsde trans}. Let $(Y, Z)$ (resp. $(\tilde{Y}, \tilde{Z})$) be a super-solution (resp. sub-solution) to \eqref{eq: bsde trans}, i.e.,
 \[
  Y + \int_0^\cdot F(s, c_s, Y_s) ds \text{ is a local supermartingale and } \tilde{Y} + \int_0^\cdot F(s, c_s, \tilde{Y}_s) ds \text{ is a local submartingale},
 \]
 with $Y_T \geq \xi \geq \tilde{Y}_T$,
 meanwhile $Z$ and $\tilde{Z}$ are determined by Doob-Meyer decomposition and martingale representation. Assuming that both $Y$ and $\tilde{Y}$ are of class $D$, then $Y\geq \tilde{Y}$. Moreover, if $Y_T > \tilde{Y}_T$, then $Y_t > \tilde{Y}_t$ for any $t\leq T$.

 To prove this comparison result,
 define
 \[
  \alpha_t:= \left\{\begin{array}{ll}\frac{F(t, c_t, Y_t) - F(t, c_t, \tilde{Y}_t)}{Y_t - \tilde{Y}_t}, & Y_t \neq \tilde{Y}_t\\ 0, & Y_t = \tilde{Y}_t \end{array}\right..
 \]
 Since $y\mapsto F(\cdot, \cdot, y)$ is decreasing, we have $\alpha \leq 0$. It then follows that $e^{\int_0^\cdot \alpha_s ds}(Y- \tilde{Y})$ is a local supermartingale, hence a supermartingale, since the exponential factor is bounded and both $Y$ and $\tilde{Y}$ are of class $D$. Therefore, $Y_T\geq\tilde{Y}_T$ implies $Y\geq\tilde{Y}$. Moreover when $Y_T > \tilde{Y}_T$, we obtain the strict comparison $Y_t > \tilde{Y}_t$ for any $t\leq T$. The uniqueness follows from the comparison result directly. Since $\gamma>1$, then $\xi=e^{-\delta \theta T} c_T^{1-\gamma} >0$. Therefore $Y>0$ follows from the strict comparison.

 Finally, we verify that $V^c$ satisfies \eqref{eq: EZ utility}. To this end, since $(Y, Z)$ solves \eqref{eq: bsde trans}, $(V^c_t, Z^c_t) = e^{\delta \theta t} (Y_t, Z_t)/(1-\gamma)$ satisfies \eqref{eq: EZ bsde}, implying that $V^c +\int_0^\cdot f(c_s, V^c_s) ds$ is a local martingale. Taking a localizing sequence $(\sigma_n)_{n\geq 1}$ for $V^c +\int_0^\cdot f(c_s, V^c_s) ds$, we obtain
 \begin{equation*}\label{eq: V^c loc}
  V^c_0 +\delta \theta \expec\bra{\int_0^{T\wedge \sigma_n} V^c_s ds}= \expec\bra{V^c_{T\wedge \sigma_n} + \int_0^{T\wedge \sigma_n} \delta \frac{c_s^{1-\frac{1}{\psi}}}{1-\frac{1}{\psi}}\pare{(1-\gamma) V^c_s}^{1-\frac{1}{\theta}} ds}.
 \end{equation*}
 Sending $n\rightarrow \infty$ on both sides, note that $V^c\leq 0$ and $\psi>1$, therefore the integrand on the left side is negative and the integrand on the right side is positive. The monotone convergence theorem and the class $D$ property of $V^c$ then yield
 \begin{equation}\label{eq: exp V^c}
  V^c_0 +\delta \theta \expec\bra{\int_0^{T} V^c_s ds}= \expec\bra{U(c_T)+ \int_0^{T} \delta \frac{c_s^{1-\frac{1}{\psi}}}{1-\frac{1}{\psi}}\pare{(1-\gamma) V^c_s}^{1-\frac{1}{\theta}} ds}.
 \end{equation}
 Since $0\geq \expec\bra{\int_0^T V^c_s ds} = \frac{1}{1-\gamma} \expec\bra{\int_0^T e^{\delta \theta s}Y_s ds} \geq \frac{1}{1-\gamma} \expec\bra{\int_0^T Y_s ds} \geq \frac{1}{1-\gamma} \int_0^T \expec[\xi] ds>-\infty$, where the second inequality holds since $\gamma>1$ and $\theta<0$, the third inequality follows from $Y_s \leq \expec_s[\xi]$ and $\gamma>1$. Subtracting $\delta \theta\expec\bra{\int_0^T V^c_s ds}$ on both sides of \eqref{eq: exp V^c}, we confirm \eqref{eq: EZ utility}.
\end{proof}

The concavity of $c\mapsto V^c$ is proved in the following. This proof utilizes simultaneously the joint concavity of the generator for \eqref{eq: bsde ord trans} and the class $D$ property of the solution to \eqref{eq: bsde trans}.

\begin{proof}[Proof of Proposition \ref{prop: EZ concave}]
 Denote the generator of \eqref{eq: bsde ord trans} as $\mathbb{F}(t, c_t, y, z) = \delta e^{-\delta t} \frac{c_t^{1-\frac{1}{\psi}}}{1-\frac{1}{\psi}} + \frac12 (\theta -1) \frac{z^2}{y}$. For $c, \tilde{c}$ and $\alpha c + (1-\alpha) \tilde{c}\in \mathcal{C}_a$, denote $\Delta X = \alpha X + (1-\alpha) \tilde{X}$, for $X=c, \bY, \bZ$ and $\tilde{X}= \tilde{c}, \tilde{\bY}, \tilde{\bZ}$, respectively. It follows from \eqref{eq: bsde ord trans} that
 \[
  d \Delta \bY_t = \bra{ -\delta e^{-\delta t} \frac{\Delta c_t^{1-\frac{1}{\psi}}}{1-\frac{1}{\psi}} - \frac12 (\theta -1) \frac{\Delta \bZ_t^2}{\Delta \bY_t} + A_t}  dt + \Delta \bZ_t dB_t,
 \]
 where, due to the concavity of $(c_t, y, z)\mapsto \mathbb{F}(t, c_t, y, z)$,
 \[
  A_t = \frac{\delta e^{-\delta t}}{1-\frac{1}{\psi}}\bra{\Delta c_t^{1-\frac{1}{\psi}} - \alpha c_t^{1-\frac{1}{\psi}} - (1-\alpha) \tilde{c}_t^{1-\frac{1}{\psi}}} + \frac12 (\theta -1) \bra{\frac{\Delta \bZ^2_t}{\Delta \bY_t} - \alpha \frac{\bZ^2_t}{\bY_t} - (1-\alpha) \frac{\tilde{\bZ}_t^2}{\tilde{\bY}_t}}\geq 0,
 \]
 and $\Delta \bY_T \leq e^{-\delta T} \Delta c_T^{1-1/\psi}/(1-1/\psi)$. Set
 \[
  \Delta Y = ((1-1/\psi) \Delta \bY)^\theta \quad \text{ and } \quad \Delta Z = (1-\gamma) ((1-1/\psi) \Delta \bY)^{\theta -1} \Delta \bZ.\]
 It\^{o}'s formula yields
 \[
  d\Delta Y_t =( -\delta \theta e^{-\delta t} \Delta c_t^{1-\frac{1}{\psi}} \Delta Y_t^{1-\frac{1}{\theta}} + (1-\gamma) \Delta Y_t^{1-\frac{1}{\theta}}A_t)\,  dt + \Delta Z_t dB_t,
 \]
 where  $(1-\gamma) \Delta Y_t^{1-1/\theta} A_t\leq 0$.
 On the other hand, $\Delta Y_T \geq e^{-\delta \theta T} \Delta c_T^{1-\gamma}$. Therefore $(\Delta Y, \Delta Z)$ is a super-solution to \eqref{eq: bsde trans}. On the other hand, $\Delta Y$ is of class $D$. Indeed, since $\theta<0$,
 \begin{equation}\label{eq: Delta Y}
 \Delta Y = \pare{\pare{1-1/\psi} \Delta \bY}^\theta \leq \alpha \pare{\pare{1-1/\psi} \bY}^\theta + (1-\alpha) (\pare{1-1/\psi} \tilde{\bY})^\theta = \alpha Y + (1-\alpha) \tilde{Y},
 \end{equation}
 where $Y$ (resp. $\tilde{Y}$) is the first component of the solution to \eqref{eq: bsde trans} with $c$ (resp. $\tilde{c}$). Therefore, $\Delta Y$ is of class $D$, because both $Y$ and $\tilde{Y}$ are. Now consider $Y^{\Delta c}$ as the first component of solution of \eqref{eq: bsde trans} where $c$ is replaced by $\Delta c$. It then follows from \eqref{eq: Delta Y} and the comparison result in Step 3 of the previous proof that
 \[
  \alpha Y_0 + (1-\alpha) \tilde Y_0 \geq \Delta Y_0 \geq Y_0^{\Delta c}.
 \]
 Dividing the previous inequality by $(1-\gamma)$ on both sides, we confirm $\alpha V^c_0 + (1-\alpha) V^{\tilde{c}}_0 \leq V^{\alpha c + (1-\alpha) \tilde{c}}_0$.
\end{proof}

\section{Proofs in Section \ref{subsec: portfolio opt}}\label{app: portfolio opt}
 Even though the generator $H$ in \eqref{eq: op bsde} has an exponential term in $y$, the parameter specification $\gamma>1$ and $\psi>1$ allow us to derive a priori bounds for $Y$. Then a solution to \eqref{eq: op bsde} is constructed via the localization technique in \cite{Briand-Hu}.

\begin{proof}[Proof of Proposition \ref{prop: existence}]
  Due to Assumption \ref{ass: op bsde} i), $\overline{W} := W- \int_0^\cdot \frac{1-\gamma}{\gamma} \rho' \sigma' \Sigma^{-1} \mu(X_s) ds$ is a $\overline{\prob}-$Brownian motion. Therefore, \eqref{eq: op bsde} can be rewritten under $\overline{\prob}$,
  and all expectations are taken with respect to $\overline{\prob}$ throughout this proof. On the other hand, recall that $\gamma>1$ and $r+ \frac{1}{2\gamma} \mu' \Sigma^{-1} \mu$ is  bounded from below. Therefore there exists a constant $h_{max}$ such that $h \leq h_{max}$. However, $\mu'\Sigma^{-1}\mu$, in many widely used models, is an unbounded function of the state variable, hence $h$ and $H(t,0,0) = h_t-\delta \theta+ \theta \frac{\delta^\psi}{\psi}$ are not bounded from below. Therefore we introduce
 \begin{equation}\label{eq: op bsde tran}
  \cY_t = \xi + \int_t^T \cH(s, \cY_s, Z_s) ds - \int_t^T Z_s d\overline{W}_s,
 \end{equation}
 where $\cY_t = Y_t + \int_0^t (h_s-\delta \theta) \,ds$, $\xi= \int_0^T (h_s-\delta \theta) \,ds$, and
 \[
  \cH(t,y,z) = \frac12 z M_t z' + \theta \frac{\delta^\psi}{\psi} e^{\frac{\psi}{\theta}\int_0^t h_s -\delta \theta \,ds} e^{-\frac{\psi}{\theta}y}.
 \]

 Consider a truncated version of \eqref{eq: op bsde tran}:
 \begin{equation}\label{eq: op trunc gamma>1}
  \cY^n_t = \xi^n + \int_t^T \cH^n(s, \cY^n_s, Z^n_s) \, ds - \int_t^T Z^n_s d\overline{W}_s,
 \end{equation}
 where $\xi^n = \int_0^T h_s \vee (-n) -\delta \theta \, ds$ is bounded and
 \[
  \cH^n(t,y,z) = \frac12 z M_t z'  + \theta \frac{\delta^\psi}{\psi} e^{\frac{\psi}{\theta} \int_0^t  h_s \vee (-n) -\delta \theta \, ds}\pare{e^{-\frac{\psi}{\theta} y}\wedge n}.
 \]
 This truncated generator $\cH^n$ is Lipschitz in $y$ and quadratic in $z$. Indeed, since eigenvalues of $\sigma' \Sigma^{-1} \sigma$ is either $0$ or $1$, $0\leq z \rho' \sigma' \Sigma^{-1} \sigma \rho z' \leq z \rho' \rho z' \leq |z|^2$. Then $\gamma>1$ and the definition of $M$ after \eqref{eq: H} implies
 \begin{equation}\label{eq: M growth}
  0< \frac{1}{\gamma} |z|^2 \leq z M(X) z' \leq |z|^2.
 \end{equation}
 Therefore it follows from \cite[Theorem 2.3]{Kobylanski} that \eqref{eq: op trunc gamma>1} admits a solution $(\cY^n, Z^n)\in \cS^\infty \times \cM^2$. Moreover, due to $\theta<0$, $\cH^n$ is decreasing in $n$. The construction of $\cY^n$ in \cite[Theorem 2.3]{Kobylanski} yields $\cY^n \geq \cY^{n+1}$. In what follows, we derive a priori bounds on $\cY^n$ uniformly in $n$. This uniform estimate facilitates the construction of a solution to \eqref{eq: op bsde tran}.

 On the one hand, $\theta<0$ and the third inequality in \eqref{eq: M growth} yield $\cH^n(t, y, z)\leq \frac12 |z|^2$. Consider
 \[
  \overline{Y}^n_t = \xi^n + \int_t^T \frac12 |\overline{Z}^n_s|^2 ds -\int_t^T \overline{Z}^n_s \, d\overline{W}_s,
 \]
 which has an explicit solution $\overline{Y}^n_t = \log \expec_t \bra{e^{\int_0^T h_s \vee (-n) -\delta \theta \, ds}}$. Then
 \begin{equation}\label{eq: oY}
  \overline{Y}^n_t - \int_0^t h_s \vee (-n) -\delta \theta \, ds = \log \expec_t \bra{e^{\int_t^T h_s \vee (-n) -\delta \theta \, ds}} \leq (h_{max}-\delta \theta) (T-t).
 \end{equation}
 On the other hand, when $y - \int_0^t h_s \vee (-n) -\delta \theta \, ds \leq (h_{max}-\delta \theta)(T-t)$, the first inequality in \eqref{eq: M growth} and $\theta<0$ imply $\cH^n(t, y, z) \geq \theta \frac{\delta^\psi}{\psi} e^{(\delta \psi-\frac{\psi}{\theta} h_{max}) T}$. Therefore consider the BSDE
 \[
  \underline{Y}_t = \xi + \theta \frac{\delta^\psi}{\psi} e^{(\delta \psi -\frac{\psi}{\theta} h_{max})T} (T-t) -\int_t^T \underline{Z}_s d\overline{W}_s,
 \]
 whose solution $\underline{Y}$ admits a representation $\underline{Y}_t = \expec_t[\xi] + \theta \frac{\delta^\psi}{\psi} e^{(\delta \psi -\frac{\psi}{\theta} h_{max}) T} (T-t)$.

 Now since $\mathcal{H}^n$ is sandwiched between two generators with simpler forms, comparison result yields 
 \begin{equation}\label{eq: oY bdd}
  \expec_t[\xi] + \theta \frac{\delta^\psi}{\psi} e^{(\delta \psi -\frac{\psi}{\theta} h_{max}) T} (T-t) = \underline{Y}_t \leq \cY^n_t \leq \overline{Y}^n_t = \log \expec_t \bra{e^{\int_0^T h_s \vee (-n) -\delta \theta \, ds}} \leq (h_{max}-\delta \theta) T,
 \end{equation}
 for any $n>0$.
 These uniform bounds on $\cY^n$ allow us to construct a solution $(\cY, Z)$ to \eqref{eq: op bsde tran} using the localization technique in \cite[Theorem 2]{Briand-Hu}; see also Step 2 in the proof of Proposition \ref{prop: EZ bsde}. The resulting $\cY$ satisfies
 \begin{equation}\label{eq: cY bdd}
  \expec_t[\xi] + \theta \frac{\delta^\psi}{\psi} e^{(\delta \psi-\frac{\psi}{\theta} h_{max}) T} (T-t) \leq \cY_t \leq  \log \expec_t \bra{e^{\xi}}.
 \end{equation}
 The previous inequalities imply that $\lim_{t\rightarrow T} \cY_t= \xi$. Hence $\cY$ satisfies the terminal condition of \eqref{eq: op bsde tran}. The desired estimates on $Y$ follows after subtracting $\int_0^t h_s -\delta \theta \, ds$ on both sides of the previous inequalities, in particular,
 \begin{equation}\label{eq: Y-ubb}
   Y_t = \mathcal{Y}_t -\int_0^t h_s -\delta \theta \,ds \leq \log \expec_t \bra{e^{\int_t^T (h_s -\delta \theta)ds}}\leq (h_{max}-\delta \theta) (T-t).
 \end{equation}For the statement on $Z$, take a localization sequence $(\sigma_n)_n$ for $\int_0^\cdot Z_s d\overline{W}_s$, \eqref{eq: op bsde tran} yields
 \[
  \frac12 \expec\bra{\int_0^{\sigma_n}Z_s M_s Z_s' ds} = \cY_0 - \expec[\cY_{\sigma_n}] - \theta \frac{\delta^\psi}{\psi} \expec\bra{\int_0^{\sigma_n} e^{\frac{\psi}{\theta} \int_0^s h_u -\delta \theta du} e^{-\frac{\psi}{\theta} \cY_s}ds}.
 \]
 Sending $n\rightarrow \infty$ on both sides, applying the second inequality in \eqref{eq: M growth} to the left-hand side, the first inequality in \eqref{eq: cY bdd} to the second term on the right-hand side, and \eqref{eq: Y-ubb} to the third term, we confirm $\expec[\int_0^T |Z_s|^2 ds]<\infty$.
\end{proof}

The following several results prepare the proofs of Theorems \ref{thm: verification} and \ref{thm: state price}. First we show $\frac{w^{1-\gamma}}{1-\gamma} e^{Y_0}$ is an upper bound for the optimal value among permissible strategies.

\begin{lem}\label{lem: value ubb}
 Let Assumption \ref{ass: op bsde} hold.
 For any permissible $(\pi, c)$,
 \begin{equation}\label{eq: value ubb}
 \frac{w^{1-\gamma}}{1-\gamma} e^{Y_0} \geq V^c_0,
 \end{equation}
 where $V^c$ is defined in Proposition \ref{prop: EZ bsde}, $Y$ is constructed in Proposition \ref{prop: existence},  and $c$ is financed by $\pi$ via \eqref{eq: wealth}.
\end{lem}

\begin{proof}
 This proof extends the technique in \cite{Hu-Imkeller-Muller} to recursive utilities. For a permissible $(\pi, c)$, define
 \[
  R^{\pi, c}_t:= \frac{(\mathcal{W}_t)^{1-\gamma}}{1-\gamma} e^{Y_t} + \int_0^t f\pare{c_s, \frac{(\mathcal{W}_s)^{1-\gamma}}{1-\gamma} e^{Y_s}} \, ds, \quad t\in[0,T],
 \]
 where $\mathcal{W}= \mathcal{W}^{\pi, c}$.
 Then \eqref{eq: R drift} and \eqref{eq: H} imply that $R$ is a local supermartingale. Due to Doob-Meyer decomposition and martingale representation, there exist an increasing process $A$  and $Z^R$ such that $R^{\pi, c}=-A + \int_0^\cdot Z^R_s dB_s$. Therefore, $\pare{\frac{(\mathcal{W})^{1-\gamma}}{1-\gamma} e^Y, Z^R}$ is a supersolution to \eqref{eq: EZ bsde}, whose terminal condition is $(\mathcal{W}_T)^{1-\gamma}/(1-\gamma)\in \mathbb{L}^1$. Indeed, since $(\mathcal{W})^{1-\gamma} e^Y$ is of class $D$ by permissibility and $Y_T=0$, we have $\expec[(\mathcal{W}_T)^{1-\gamma}]<\infty$. On the other hand, consider the utility $V^0_c$ associated to the consumption stream $c$ and the terminal lump sum $\mathcal{W}_T$. The comparison result in the proof of Proposition \ref{prop: EZ bsde} confirms \eqref{eq: value ubb}.
\end{proof}

In what follows we will show that $(\pi^*, c^*)$ is a permissible strategy and it attains the upper bound $\frac{w^{1-\gamma}}{1-\gamma} e^{Y_0}$. First, we establish an important result that certain exponential local martingale associated to $\pi^*$ is a martingale.

\begin{lem}\label{lem: martingale}
 Let Assumptions \ref{ass: coeff}, \ref{ass: op bsde} and \ref{ass: phi} hold. Then $Q:= \mathcal{E}\pare{\int (1-\gamma) (\pi^*_s)' \sigma_s dW^\rho_s + \int Z_s dW_s}$ is a $\prob-$martingale on $[0,T]$.
\end{lem}

\begin{proof}
 It follows from \eqref{eq: opt st}, the definition of $W^\rho$ and $M$ that
 \begin{align*}
  (1- \gamma) (\pi^*)' \sigma dW^\rho + Z dW &= \pare{\frac{1-\gamma}{\gamma} \mu' \Sigma^{-1} \sigma \rho + Z M} dW +  \frac{1-\gamma}{\gamma} \pare{\mu' + Z \rho' \sigma'} \Sigma^{-1} \sigma \rho^{\bot} dW^\bot\\
  &=: L^{(1)} dW + L^{(2)} dW^\bot.
 \end{align*}
 Here we suppress time subscripts to simplify notation. First we claim that
 if $Q^{(1)}:=\mathcal{E}(\int L^{(1)}_sdW_s)$ is a martingale, so is $Q$. Indeed, for any $t\leq T$,
 \begin{equation}\label{eq: mart}
 \begin{split}
  \expec[Q_t]&= \expec\bra{\mathcal{E}\pare{\int L^{(1)}_s dW_s}_t \mathcal{E}\pare{\int L^{(2)}_s dW^\bot_s}_t}\\
  &= \expec\bra{\mathcal{E}\pare{\int L^{(1)}_s dW_s}_t \expec\bra{\left.\mathcal{E}\pare{\int L^{(2)}_s dW^\bot_s}_t\right| \F^W}}\\
  &= \expec\bra{\mathcal{E}\pare{\int L^{(1)}_s dW_s}_t}\\
  &=1.
 \end{split}
 \end{equation}
 Here $\F^W= \sigma(W_s; 0\leq s\leq T)$, the third identity follows from \cite[Lemma 4.8]{Karatzas-Kardaras} since $L^{(2)}$ and $W^\bot$ are independent, and the fourth identity is due to the martingale assumption on $Q^{(1)}$. In the remaining of the proof, we will prove the martingale property of $Q^{(1)}$.

 For the sequence of subdomains $(E_n)_n$ in Assumption \ref{ass: phi} i), define $\tau_n:= \inf\{t\geq 0\,|\, X_t \notin E_n\}\wedge T$. we first prove that $Y_{\cdot \wedge \tau_n}$ is bounded. Since we have seen in Proposition \ref{prop: existence} that $Y$ is bounded from above, it suffices to show  $\expec^{\overline{\prob}}_{\cdot \wedge \tau_n}\bra{\int_{\cdot \wedge \tau_n}^T h_s ds}$ is bounded from below. Then \eqref{eq: Y bdd} implies that $Y_{\cdot \wedge \tau_n}$ is bounded as well. Due to the Markovian structure, define
 \[
  y(t,x) := \expec^{\overline{\prob}}\bra{\left. \int_t^T h(X_s) \, ds \right| X_t = x}.
 \]
 The Feynman-Kac formula (see \cite{Heath-Schweizer} when the equation is not uniformly parabolic) implies that, under Assumption \ref{ass: coeff}, $y\in C^{1,2}([0,T]\times E)$ and it is the unique solution to
 \[
  \partial_t y + \mathcal{L} y + h =0, \quad y(T, x)=0,
 \]
 where $\mathcal{L}$ is the infinitesimal generator of $X$ under $\overline{\prob}$. Now since $\overline{E}_n$ is compact, the continuity of $y$ implies that $y(\cdot \wedge \tau_n, X_{\cdot \wedge \tau_n})$ is bounded.

 As a solution to \eqref{eq: op bsde}, $(Y, Z)$ satisfies
 \[
  Y_t = Y_{\tau_n} + \int_t^{\tau_n} H(s, Y_s, Z_s) ds -\int_t^{\tau_n} Z_s \, dW_s, \quad t\in[0, \tau_n].
 \]
 Since both $X_{\cdot \wedge \tau_n}$ and $Y_{\cdot \wedge \tau_n}$ are bounded, it follows from the \BMO-estimate for quadratic BSDEs (cf. eg. \cite[Lemma 3.1]{Morlais}) that $\int_0^{\cdot\wedge \tau_n} Z_s dW_s$ is a \BMO-martingale. Note that both $\mu' \Sigma^{-1} \sigma \rho(X_{\cdot \wedge \tau_n})$ and $M(X_{\cdot \wedge \tau_n})$ are bounded. Therefore $\int_0^{\cdot\wedge \tau_n} L^{(1)}_s dW_s$ is a \BMO-martingale as well. Then \cite[Theorem 2.3]{Kazamaki} implies that $\mathcal{E}(\int L^{(1)}_s dW_s)_{\cdot \wedge \tau_n}$ is a martingale. Therefore $d\qprob^n/d\prob := \mathcal{E}(\int L^{(1)}_s dW_s)_{\tau_n}$ defines $\qprob^n$ on $\F_{\tau_n}$ which is equivalent to $\prob$.

 Assuming that $\lim_{n\rightarrow \infty} \qprob^n(\tau_n < T)=0$, by the monotone convergence theorem,
 \begin{align*}
  \expec\bra{\mathcal{E}\pare{\int L^{(1)}_s dW_s}_{T}} &= \lim_{n\rightarrow \infty} \expec\bra{\mathcal{E}\pare{\int L^{(1)}_s dW_s}_{\tau_n} \indic_{\{\tau_n =T\}}} \\
  &= \lim_{n\rightarrow \infty} \expec\bra{\mathcal{E}\pare{\int L^{(1)}_s dW_s}_{\tau_n}} - \lim_{n\rightarrow \infty} \expec\bra{\mathcal{E}\pare{\int L^{(1)}_s dW_s}_{\tau_n}\indic_{\{\tau_n <T\}}}\\
  &= 1- \lim_{n\rightarrow \infty} \qprob^n(\tau_n < T) \\
  &=1,
 \end{align*}
 proving the martingale property of $\mathcal{E}(\int L^{(1)}_s dW_s)$ on $[0,T]$.

 It remains to prove $\lim_{n\rightarrow \infty} \qprob^n(\tau_n < T)=0$. To this end, \eqref{eq: op bsde} yields
 \[
  Y_t = Y_0 -\int_0^t H(s, Y_s, Z_s)ds + \int_0^t Z_s dW_s.
 \]
 On the other hand, recall $\fF$ from \eqref{eq: fF}, we have from It\^{o}'s formula,
 \begin{align*}
  \phi(X_t) &= \phi(x) + \int_0^t b' \nabla \phi(X_s) + \frac12 \sum_{i,j=1}^k A_{ij} \partial^2_{x_i x_j} \phi(X_s) \, ds + \int_0^t \nabla \phi' a(X_s) dW_s\\
  &= \phi(x) + \int_0^t \bra{\fF[\phi] - \frac12 \nabla \phi' a M a' \nabla \phi - h  - \frac{1-\gamma}{\gamma} \mu'\Sigma^{-1} \sigma \rho a' \nabla \phi} ds  + \int_0^t \nabla \phi' a(X_s) dW_s
 \end{align*}
 Taking difference of the previous two identities,
 \begin{align*}
  &Y_t - \phi(X_t) \\
  &= Y_0 - \phi(x) + \int_0^t \pare{Z_s - \nabla \phi' a(X_s)} dW_s\\
  &\quad - \int_0^t \bra{\frac12 Z M Z' - \frac12 \nabla \phi'a M a' \nabla \phi + \theta \frac{\delta^\psi}{\psi} e^{-\frac{\psi}{\theta} Y_t} -\delta \theta + \fF[\phi] + \frac{1-\gamma}{\gamma} \mu' \Sigma^{-1} \sigma \rho(Z-\nabla \phi' a)'} ds\\
  &= Y_0 - \phi(x) + \int_0^t \pare{Z_s - \nabla \phi' a(X_s)} d W^n_s\\
  & \quad - \int_0^t \bra{\frac12 Z M Z' - \frac12  \nabla \phi' a M a' \nabla \phi - (Z- \nabla \phi' a) M Z' +\theta \frac{\delta^\psi}{\psi} e^{-\frac{\psi}{\theta} Y_t} -\delta \theta + \fF[\phi]} ds\\
  &= Y_0 - \phi(x) + \int_0^t \pare{Z_s - \nabla \phi' a(X_s)} d W^n_s\\
  &\quad +\int_0^t \bra{\frac12 (Z- \nabla \phi' a) M (Z'- a\nabla \phi) - \theta \frac{\delta^\psi}{\psi} e^{-\frac{\psi}{\theta} Y_t} + \delta \theta - \fF[\phi]} ds, \quad t\leq \tau_n,
 \end{align*}
 where $W^n := W - \int_0^\cdot L^{(1)}_s ds$ is a $\qprob^n-$Brownian motion on $[0, \tau_n]$. On the right hand side, the quadratic term is nonnegative, $-\theta \frac{\delta^\psi}{\psi} e^{-\frac{\psi}{\theta} Y_t}$ is nonnegative since $\theta<0$, and $\delta \theta - \fF[\phi]$ is also bounded from below due to Assumption \ref{ass: phi} ii). Therefore, there exists some negative constant $C$ such that
 \begin{equation}\label{eq: Y taun lbb}
  Y_{\tau_n} - \phi(X_{\tau_n}) \geq Y_0 - \phi(x) + C \tau_n + \int_0^{\tau_n} (Z_s - \nabla \phi' a) dW^n_s.
 \end{equation}
 The stochastic integral on the right hand side has zero expectation under $\qprob^n$. Indeed, since $\int_0^{\cdot \wedge \tau_n} Z_s dW_s$ is a $\BMO(\prob)-$martingale and $\nabla \phi' a(X_{\cdot \wedge \tau_n})$ is bounded, hence $\int_0^{\cdot \wedge \tau_n}(Z_s - \nabla\phi' a(X_s)) dW_s$ is a $\BMO(\prob)-$martingale as well. Now since $\int_0^{\cdot \wedge \tau_n} L^{(1)}_s dW_s$ is a $\BMO(\prob)-$martingale, \cite[Theorem 3.6]{Kazamaki} implies that $\int_0^{\cdot \wedge \tau_n} (Z_s - \nabla \phi' a(X_s)) dW^n_s$ is a $\BMO(\qprob^n)-$martingale. Therefore its expectation under $\mathbb{Q}^n$ is zero. It then follows from \eqref{eq: Y taun lbb} that
 \begin{equation}\label{eq: exp Qn}
 \expec^{\qprob^n}\bra{Y_{\tau_n} -\phi(X_{\tau_n})}  \geq Y_0 - \phi(x) + CT>-\infty, \quad \text{ for each } n.
 \end{equation}
 Now since $Y$ is bounded from above and $\phi$ is bounded from below due to Assumption \ref{ass: phi} i), there exists a constant $C$, such that
 \[Y_{\tau_n} - \phi(X_{\tau_n}) = \pare{Y_{\tau_n} - \phi(X_{\tau_n})}\indic_{\{\tau_n <T\}} + \pare{Y_T - \phi(X_T)}\indic_{\{\tau_n = T\}} \leq C - \inf_{x\in \partial E_n} \phi(x) \,\indic_{\{\tau_n <T\}}.\] Now sending $n\rightarrow \infty$ in \eqref{eq: exp Qn}, Assumption \ref{ass: phi} i) and the previous inequality confirm that $\lim_{n\rightarrow \infty} \qprob^n(\tau_n < T)=0$.
\end{proof}

The martingale property in the previous result helps to verify the permissibility of $(\pi^*, c^*)$.

\begin{cor}\label{cor: opt class D}
 Let Assumptions \ref{ass: coeff}, \ref{ass: op bsde} and \ref{ass: phi} hold. Then $\pare{\mathcal{W}^{*}}^{1-\gamma} e^Y$ is of class $D$ on $[0,T]$, where $\mathcal{W}^*$ is the wealth process associated to $(\pi^*, c^*)$. 
\end{cor}

\begin{proof}
 The calculation leading to \eqref{eq: H} yields
 \begin{align*}
  &d\pare{\mathcal{W}^*_t}^{1-\gamma} e^{Y_t} \\
  &= - \pare{\mathcal{W}^*_t}^{1-\gamma} e^{Y_t} \pare{\delta \theta \pare{c^*_s}^{1-\frac{1}{\psi}} \pare{\pare{\mathcal{W}^*_t}^{1-\gamma} e^{Y_t}}^{-\frac{1}{\theta}}- \delta \theta} dt + \pare{\mathcal{W}^*_t}^{1-\gamma} e^{Y_t} \bra{(1-\gamma) (\pi^*_t)' \sigma_t dW^\rho_t + Z_t dW_t}\\
  &= -\pare{\mathcal{W}^*_t}^{1-\gamma} e^{Y_t} \bra{\theta \delta^\psi e^{-\frac{\psi}{\theta}Y_t} -\delta \theta} dt + \pare{\mathcal{W}^*_t}^{1-\gamma} e^{Y_t} \bra{(1-\gamma) (\pi^*_t)' \sigma_t dW^\rho_t + Z_t dW_t},
 \end{align*}
 where the second identity follows from the form of $c^*$ in \eqref{eq: opt st}. Therefore,
 \[
  \pare{\mathcal{W}_t^*}^{1-\gamma} e^{Y_t} = w^{1-\gamma} e^{Y_0} \exp \pare{-\int_0^t \pare{\delta^\psi \theta e^{-\frac{\psi}{\theta} Y_s} -\delta \theta}ds} \mathcal{E} \pare{\int (1-\gamma) (\pi^*_s)' \sigma_s dW^\rho_s + \int Z_s dW_s}_t.
 \]
 Since $\theta<0$ and $Y$ is bounded from above, the second exponential term on the right is bounded, uniformly in $t$. Meanwhile, due to Lemma \ref{lem: martingale}, the stochastic exponential on the right is of class $D$ on $[0,T]$. The statement is then confirmed.
\end{proof}

\begin{lem}\label{lem: c admissible}
 Let Assumptions \ref{ass: coeff}, \ref{ass: op bsde}, \ref{ass: phi}, and \ref{ass: Q^0} hold. Let $c^*$ be in \eqref{eq: opt st} and $c_T = \mathcal{W}^{*}_T$. Then $c^* \in \C_a$.
\end{lem}
\begin{proof}
 Since $Y_T=0$, the class $D$ property of $\pare{\mathcal{W}^*}^{1-\gamma} e^Y$ in Corollary \ref{cor: opt class D} yields $\expec[\pare{\mathcal{W}_T^{*}}^{1-\gamma}]<\infty$. On the other hand, the expression of $\tilde{c}^*$ in \eqref{eq: opt st} implies
 \[
  e^{-\delta s} (c^*_s)^{1-\frac{1}{\psi}} = e^{-\delta s} \delta^{\psi - 1} e^{- \frac{\psi-1}{\theta} Y_s} \pare{\mathcal{W}_s^{*}}^{1-\frac{1}{\psi}}.
 \]
 Since $\psi>1$, $\theta<0$, and $Y$ is bounded from above, the first three terms on the right hand side are bounded. Therefore it suffices to prove
 \begin{equation}\label{eq: W int}
  \expec\bra{\int_0^T \pare{\mathcal{W}_s^{*}}^{1-\frac{1}{\psi}} ds} <\infty.
 \end{equation}

 To this end, it follows from Assumption \ref{ass: Q^0} that
 \begin{align*}
  \expec\bra{\int_0^T \pare{\mathcal{W}_s^{*}}^{1-\frac{1}{\psi}} ds}  & = \int_0^T \expec^{\qprob^0}\bra{e^{\pare{1-\frac1\psi}\int_0^s r_u du} \mathcal{E}\pare{\int \lambda_u' dW_u^0}_T e^{-\pare{1-\frac{1}{\psi}}\int_0^s r_u du}\pare{\mathcal{W}^*_s}^{1-\frac{1}{\psi}}} ds\\
  & \leq \int_0^T \expec^{\qprob^0}\bra{e^{\pare{1-\frac1\psi}\int_0^T(r_u)_+ du} \mathcal{E}\pare{\int \lambda_u' dW_u^0}_T e^{-\pare{1-\frac1\psi}\int_0^s r_u du}\pare{\mathcal{W}^*_s}^{1-\frac{1}{\psi}}}ds\\
  &\leq \expec^{\qprob^0}\bra{e^{(\psi-1) \int_0^T (r_u)_+ du} \mathcal{E} \pare{\int \lambda_u' dW^0_u}^\psi_T}^{\frac1\psi} \int_0^T \expec^{\qprob^0}\bra{e^{-\int_0^s r_u du} \mathcal{W}^*_s}^{1-\frac1\psi}ds\\
  &\leq w^{1-\frac1\psi} T \,\expec^{\qprob^0}\bra{e^{(\psi-1) \int_0^T (r_u)_+ du} \mathcal{E} \pare{\int \lambda_u' dW^0_u}^\psi_T}^{\frac1\psi}\\
  &<\infty.
 \end{align*}
 Here the first inequality follows from $\psi>1$; the second inequality holds due to H\"{o}lder's inequality; the third inequality is obtained using the fact that $e^{-\int_0^\cdot r_s ds} \mathcal{W}^*$ is a nonnegative $\qprob^0-$local martingale, hence a $\qprob^0-$supermrtingale; and the fourth inequality holds thanks to \eqref{eq: Q^0 int}.
\end{proof}

Now we are ready to prove the first main result.

\begin{proof}[Proof of Theorem \ref{thm: verification}]
 Corollary \ref{cor: opt class D} and Lemma \ref{lem: c admissible} have already shown that $(\pi^*, c^*)$ is permissible. Choosing $(\pi^*, c^*)$, we have from \eqref{eq: R drift}, \eqref{eq: H} and $Y_T=0$ that
 \[
  \frac{(\mathcal{W}^*_t)^{1-\gamma}}{1-\gamma} e^{Y_t} = \frac{(\mathcal{W}^*_T)^{1-\gamma}}{1-\gamma} + \int_t^T f\pare{c^*_s, \frac{(\mathcal{W}^*_s)^{1-\gamma}}{1-\gamma}} ds -\int_t^T Z_s dB_s,
 \]
 for some $Z$. Then the class $D$ property of $\pare{\mathcal{W}^*}^{1-\gamma} e^Y$ and Proposition \ref{prop: existence} combined imply
 \begin{equation}\label{eq: V* resp}
  \frac{w^{1-\gamma}}{1-\gamma} e^{Y_0} = \expec\bra{\int_0^T f\pare{c^*_s, \frac{(\mathcal{W}_s^*)^{1-\gamma}}{1-\gamma}e^{Y_t}} \,ds + \frac{(\mathcal{W}_T^*)^{1-\gamma}}{1-\gamma}}.
 \end{equation}
 Therefore the upper bound in Lemma \ref{lem: value ubb} is attained by $(\pi^*, c^*)$.
\end{proof}

Finally, we prove Lemma \ref{lem: deflator} and Theorem \ref{thm: state price}.

\begin{proof}[Proof of Lemma \ref{lem: deflator}]
 Calculation using \eqref{eq: wealth} and \eqref{eq: D sde} shows that $\mathcal{W} D^* + \int_0^\cdot D^*_s c_s ds$ is a local martingale.
 It then remains to prove \eqref{eq: D sde}. To ease notation, suppress all time subscripts. Using \eqref{eq: op bsde} and \eqref{eq: opt st}, calculation shows
 \begin{align*}
  d (\mathcal{W}^*)^{-\gamma} =& (\mathcal{W}^*)^{-\gamma} \bra{-\gamma(r  - \tilde{c}^* + (\pi^*)' \mu) + \frac{\gamma(\gamma+1)}{2} (\pi^*)' \Sigma \pi^*}dt - \gamma (\mathcal{W}^*)^{-\gamma} (\pi^*)' \sigma dW^\rho\\
  =& (\mathcal{W}^*)^{-\gamma} \bra{-\gamma (r -\tilde{c}^*) + \frac{1-\gamma}{2\gamma} \mu' \Sigma^{-1} \mu + \frac{1}{\gamma} \mu' \Sigma^{-1} \sigma \rho Z' + \frac{1+\gamma}{2\gamma} Z \rho' \sigma' \Sigma^{-1} \sigma \rho Z'} dt \\
  &- \gamma (\mathcal{W}^*)^{-\gamma} (\pi^*)' \sigma dW^\rho\\
  de^Y =& e^Y \bra{-H(t, Y, Z) + \frac12 ZZ'} dt + e^Y Z dW.
 \end{align*}
 Combining the previous two identities, \eqref{eq: state price tran}, and the expression for $\tilde{c}^*$ in \eqref{eq: opt st}, we confirm
 \begin{align*}
  dD^* = & D^* \left[-\gamma (r -\tilde{c}^*) + (\theta -1) \delta^\psi e^{-\frac{\psi}{\theta} Y} - \delta \theta\right.\\
  & \hspace{7mm} \left.+ \frac{1-\gamma}{\gamma} \mu' \Sigma^{-1}\mu + \frac{1-\gamma}{\gamma} \mu' \Sigma^{-1} \sigma \rho Z' + \frac12 Z M Z' - H(t, Y, Z)\right] dt\\
  &+ D^* \bra{-\gamma (\pi^*)' \sigma dW^\rho + Z dW}\\
  =& D^* \bra{-r + \pare{\theta -1 - \frac{\theta}{\psi} + \gamma} \delta^\psi e^{-\frac{\psi}{\theta} Y}} dt + D^* \bra{-\gamma (\pi^*)' \sigma dW^\rho + Z dW}\\
  =& -r D^* dt + D^* \bra{-\gamma (\pi^*)' \sigma dW^\rho + Z dW},
 \end{align*}
 where the third identity follows from $\theta + \gamma -1 - \frac{\theta}{\psi} = 0$.
\end{proof}

\begin{proof}[Proof of Theorem \ref{thm: state price}]
 It follows from \eqref{eq: opt st} and \eqref{eq: state price tran} that
 \begin{equation}\label{eq: WD+int}
  \mathcal{W}^*_t D^*_t + \int_0^t D^*_s c^*_s ds = C_t  (\mathcal{W}_t^*)^{1-\gamma} e^{Y_t} + \int_0^t C_s  \delta^\psi e^{-\frac{\psi}{\theta} Y_s} (\mathcal{W}_s^*)^{1-\gamma} e^{Y_s} \, ds.
 \end{equation}
 Here $C_t= w^{\gamma}e^{-Y_0} \exp\bra{\int_0^t (\theta-1) \delta^{\psi} e^{-\frac{\psi}{\theta} Y_u} du-\delta \theta t}$, $t\in [0,T]$. Since $\theta<0$, $C$ is bounded from above by a constant. We have already seen in Lemma \ref{lem: deflator} that $\mathcal{W}^* D^* + \int_0^\cdot D^*_s c^*_s ds$ is a nonnegative local martingale. It suffices to prove that it is of class $D$. To this end, it follows from \eqref{eq: V* resp} that
 \begin{align*}
  &\expec\bra{\int_0^T \delta \frac{(c^*_s)^{1-\frac{1}{\psi}}}{1-\frac{1}{\psi}} \pare{\pare{\mathcal{W}^*_s}^{1-\gamma}e^{Y_s}}^{1-\frac{1}{\theta}}ds}\\
  &= \frac{w^{1-\gamma}}{1-\gamma} e^{Y_0} - \frac{1}{1-\gamma} \expec\bra{\pare{\mathcal{W}^*_T}^{1-\gamma}} + \frac{\delta}{1-\frac{1}{\psi}} \int_0^T\expec\bra{ \pare{\mathcal{W}^*_s}^{1-\gamma} e^{Y_s} }ds\\
  &<\infty.
 \end{align*}
 Here since $\pare{\mathcal{W}^*}^{1-\gamma} e^Y$ is of class $D$, $\expec\bra{ \pare{\mathcal{W}^*_s}^{1-\gamma} e^{Y_s} }$ is bounded uniformly in $s$. Therefore the previous inequality holds. On the other hand, using the expression of $c^*$ in \eqref{eq: opt st},
 \[
  \expec\bra{\int_0^T \delta \frac{(c^*_s)^{1-\frac{1}{\psi}}}{1-\frac{1}{\psi}} \pare{\pare{\mathcal{W}^*_s}^{1-\gamma}e^{Y_s}}^{1-\frac{1}{\theta}}ds} = \frac{\delta^\psi}{1-\frac{1}{\psi}} \expec \bra{\int_0^T \pare{\mathcal{W}^*_s}^{1-\gamma} e^{(1-\frac{\psi}{\theta}) Y_s} ds}.
 \]
 Then $\psi>1$ and the previous two equations combined yield that the second term on the right hand side of \eqref{eq: WD+int} is bounded from above by an integrable random variable, hence is of class $D$. Meanwhile, using the class $D$ property of $\pare{\mathcal{W}^*}^{1-\gamma} e^Y$ again, the first term on the right of \eqref{eq: WD+int} is also of class $D$. This confirms the class $D$ property of $\mathcal{W}^* D^* + \int_0^\cdot D^*_s c^*_s ds$.
\end{proof}

\section{Proofs in Section \ref{sec: application}}
To prove Proposition \ref{prop: Heston}, let us recall the following result on the Laplace transform of integrated square root process; cf. \cite[Equation (2.k)]{Pitman-Yor} or \cite[Equation (3.2)]{CGMY}.

\begin{lem}\label{lem: int CIR}
 Consider $X$ with dynamics
 \[
  dX_t = (\vartheta- \kappa X_t)dt + a \sqrt{X_t} dW_t,
 \]
 where $W$ is a $1-$dimensional Brownian motion. When
 \[
  q < \frac{\kappa^2}{2a^2},
 \]
 the Laplace transform
 \[
  \expec\bra{\left.\exp\pare{q\int_0^T X_s ds} \right| X_0=x}
 \]
 is well-defined for any $T\geq 0$.
\end{lem}

\begin{proof}[Proof of Proposition \ref{prop: Heston}]
  Assumptions \ref{ass: op bsde}, \ref{ass: phi}, and \ref{ass: Q^0} are verified in what follows. We denote $\sigma(x) = \sqrt{x}\sigma$, $\Sigma(x)= x \Sigma$, $b(x) = b(\ell - x)$, $a(x) = a\sqrt{x}$, and $\Theta = \sigma' \Sigma^{-1} \sigma$.

  \vspace{2mm}

  \noindent{\underline{Assumption \ref{ass: op bsde}}:} Note $\frac{1-\gamma}{\gamma} \mu'(x) \Sigma^{-1}(x) \sigma(x) \rho(x) =\frac{1-\gamma}{\gamma}  \lambda' \Theta \rho \sqrt{x}$. Consider the martingale problem associated to $\overline{\mathcal{L}}:= \bra{b\ell - \pare{b-\frac{1-\gamma}{\gamma} a \lambda' \Theta\rho} x}\partial_x + \frac12 a^2 x \partial^2_x$ on $(0,\infty)$. Since $b\ell > \frac12 a^2$, Feller's test of explosion implies that the previous martingale problem is well-posed. Then \cite[Remark 2.6]{Cheridito-Filipovic-Yor} implies that the stochastic exponential in Assumption \ref{ass: op bsde} i) is a $\prob-$martingale, hence $\overline{\prob}$ is well defined. For Assumption \ref{ass: op bsde} ii), $h(x) = (1-\gamma) r_0 + \bra{(1-\gamma) r_1 + \frac{1-\gamma}{2\gamma} \lambda'\Theta\lambda}x$. Since $X$ has the following dynamics under $\overline{\prob}$:
  \[
   dX_t = \bra{b\ell - \pare{b- \frac{1-\gamma}{\gamma}a \lambda' \Theta\rho}X_t} + a \sqrt{X_t} d\overline{W}_t,
  \]
  where $\overline{W}$ is a $\overline{\prob}-$Brownian motion. Then $\expec^{\overline{\prob}}[\int_0^T h(X_s) ds]>-\infty$ follows from the fact that $\expec^{\overline{\prob}}[X_s]$ is bounded uniformly for $s\in[0,T]$.

  \vspace{2mm}

  \noindent{\underline{Assumption \ref{ass: phi}}:} The operator $\fF$ in \eqref{eq: fF} reads
  \[
   \fF[\phi] = \frac12 a^2 x \partial^2_x \phi + \pare{b\ell - bx + \frac{1-\gamma}{\gamma} a \lambda' \Theta\rho x} \partial_x \phi +\frac12 \tilde{M} a^2 x (\partial_x \phi)^2 + (1-\gamma) (r_0+ r_1 x) + \frac{1-\gamma}{2\gamma} \lambda'\Theta\lambda x,
  \]
  where $\tilde{M} = 1+ \frac{1-\gamma}{\gamma} \rho' \Theta\rho>0$. Consider $\phi(x) = -\underline{c} \log x + \overline{c} x$, for two positive constants $\underline{c}$ and $\overline{c}$ determined later. It is clear that $\phi(x)\uparrow \infty$ when $x\downarrow 0$ or $x\uparrow \infty$. On the other hand, calculation shows
  \begin{align*}
   \fF[\phi] =& C +  \bra{\frac12 a^2\underline{c} + \frac12 a^2\underline{c}^2 \tilde{M} - b\ell \underline{c}} \frac1x \\
   &+ \bra{-\pare{b- \frac{1-\gamma}{\gamma} a \lambda' \Theta\rho} \overline{c} + \frac12 a^2 \overline{c}^2 \tilde{M} + (1-\gamma) r_1 + \frac{1-\gamma}{2\gamma} \lambda' \Theta\lambda}x,
  \end{align*}
  where $C$ is a constant. Since $b \ell >\frac12 a^2$, the coefficient of $1/x$ is negative for sufficiently small $\underline{c}$. When $r_1$ or $\lambda'\Theta\lambda>0$, since $\gamma>1$, the coefficient of $x$ is negative for sufficiently small $\overline{c}$. Therefore, these choices of $\underline{c}$ and $\overline{c}$ imply that $\fF[\phi](x) \downarrow -\infty$ when $x\downarrow 0$ or $x\uparrow \infty$, hence $\fF[\phi]$ is bounded from above on $\Real$, verifying Assumption \ref{ass: phi}.

  \vspace{2mm}
  \noindent{\underline{Assumption \ref{ass: Q^0}}:} Consider the martingale problem associated to $\mathcal{L}^0:= \bra{b\ell - bx - a \rho' \lambda x} \partial_x + + \frac12 a^2 x \partial^2_x$ on $(0,\infty)$. Since $b\ell >\frac12 a^2$, Feller's test of explosion implies that this martingale problem is well-posed and its solution, denoted by $\qprob^\rho$, satisfies $\frac{d\qprob^\rho}{d\prob} = \mathcal{E}\pare{\int-\lambda' \rho \sqrt{X_s} dW_s}_T$. Define $\qprob^0$ via
  \[
   \frac{d\qprob^0}{d\prob} :=  \mathcal{E}\pare{-\int\lambda' \rho \sqrt{X_s} dW_s-\int\lambda' \rho^\bot \sqrt{X_s} dW_s^\bot}_T= \mathcal{E}\pare{\int -\lambda' \sqrt{X_s} dW^\rho_s}_T.
  \]
  Here, due to the independence between $X$ and $W^\bot$, proof similar to \eqref{eq: mart} implies that both stochastic exponentials on the right are $\prob-$martingales; hence $\qprob^0$ is well defined, and $\lambda$ in Assumption \ref{ass: Q^0} can be chosen as $\lambda\sqrt{X}$.

  To verify \eqref{eq: Q^0 int}, note
  \begin{equation}\label{eq: exp^psi}
   \mathcal{E}\pare{\int \lambda' \sqrt{X}_s dW^0_s }^\psi_T = \exp\pare{\frac12 (\psi^2-\psi) \lambda' \lambda \int_0^T  X_s ds} \mathcal{E}\pare{\int \psi\lambda'\sqrt{X_s} dW^0_s}_T,
  \end{equation}
  where  $W^0 := W^\rho + \int_0^\cdot \lambda \sqrt{X_s}ds$ is a $\qprob^0-$Brownian motion.
  Following the construction of $\qprob^0$, one can similarly show $\mathcal{E}\pare{\int \psi\lambda'\sqrt{X_s} dW^0_s}$ is a $\qprob^0-$martingale. Hence $\qprob^\psi$ can be defined via
  \[
   \frac{d\qprob^\psi}{d\qprob^0} := \mathcal{E}\pare{\int \psi\lambda'\sqrt{X_s} dW^0_s}_T.
  \]
  Combining the previous two change of measures, the dynamics of $X$ can be rewritten as
  \[
   dX_t = \bra{b\ell - \pare{b- (\psi-1)a \lambda' \rho} X_t} dt + a \sqrt{X_t} dW^\psi_t,
  \]
  where $W^\psi := W + \int_0^\cdot(1-\psi) \lambda' \rho \sqrt{X_s}ds$ is a $1-$dimensional $\qprob^\psi-$Brownian motion. On the other hand, calculation using \eqref{eq: exp^psi} shows
  \[
    \expec^{\qprob^0}\bra{e^{(\psi-1) \int_0^T r_s ds} \mathcal{E}\pare{\int \eta'_s dB^{\qprob^0}_s}_T^\psi} = e^{(\psi-1) r_0 T} \expec^{\qprob^\psi} \bra{\exp\pare{\bra{(\psi-1) r_1 + \frac12 (\psi^2 -\psi)\lambda' \lambda} \int_0^TX_s ds}}.
  \]
  Then Lemma \ref{lem: int CIR} implies that the expectation on the right hand side is finite when
  \[
   (\psi-1) r_1 + \frac12 (\psi^2-\psi)\lambda' \lambda < \frac{\pare{b-(\psi-1)a \lambda' \rho}^2}{2a^2}.
  \]
  This is exactly the assumption in Proposition \ref{prop: Heston} ii).
\end{proof}

\begin{proof}[Proof of Proposition \ref{prop: OU}]
 Assumptions \ref{ass: op bsde}, \ref{ass: phi}, and \ref{ass: Q^0} are verified. Then statements of Theorems \ref{thm: verification} and \ref{thm: state price} follow. We denote $\Theta = \sigma' \Sigma^{-1} \sigma$ throughout the proof to simplify notation.

 \vspace{2mm}

 \noindent{\underline{Assumption \ref{ass: op bsde}}:} Note $\frac{1-\gamma}{\gamma} \mu'(x) \Sigma^{-1}(x) \sigma(x) \rho(x) =\frac{1-\gamma}{\gamma}  (\lambda_0 + \lambda_1 x)' \Theta \rho$. Consider the martingale problem associated to $\overline{\mathcal{L}}:=\bra{-bx + \frac{1-\gamma}{\gamma}a (\lambda_0 + \lambda_1 x)'\Theta \rho} \partial_x + \frac12 a^2 \partial^2_x$ on $\Real$. This martingale problem is well-posed since all coefficients of $\overline{\mathcal{L}}$ have at most linear growth. Then \cite[Remark 2.6]{Cheridito-Filipovic-Yor} implies that the stochastic exponential in Assumption \ref{ass: op bsde} is a $\prob-$martingale, hence $\overline{\prob}$ is well defined. For Assumption \ref{ass: op bsde} ii), $h(x) = (1-\gamma) (r_0+r_1 x) + \frac{1-\gamma}{2\gamma} (\lambda_0 + \lambda x)' \Theta (\lambda_0+ \lambda_1 x)$ is bounded from below when either $r_1=0$ or $\lambda_1' \Theta \lambda_1 >0$. Since $X$ is another Ornstein-Uhlenbeck process, with modified linear drift, under $\overline{\prob}$, then $X$ has all finite moments, cf. \cite[Chapter 5, Equation (3.17)]{Karatzas-Shreve-BM}, then Assumption \ref{ass: op bsde} ii) is satisfied.

 \vspace{2mm}

 \noindent{\underline{Assumption \ref{ass: phi}}:} The operator $\fF$ in \eqref{eq: fF} reads
 \[
 \begin{split}
  \fF[\phi]=& \frac12 a^2 \partial^2_{x}\phi + \pare{-bx + \frac{1-\gamma}{\gamma} a (\lambda_0 + \lambda_1 x)' \Theta \rho} \partial_x \phi + \frac12 a^2 \tilde{M}  (\partial_x \phi)^2\\
  &+ (1-\gamma) (r_0 + r_1 x) + \frac{1-\gamma}{2\gamma} (\lambda_0 + \lambda_1 x)' \Theta (\lambda_0 + \lambda_1 x),
 \end{split}
 \]
 where $\tilde{M} = 1+ \frac{1-\gamma}{\gamma} \rho' \Theta \rho>0$. Consider $\phi(x) = c x^2$, for a positive constant $c$ determined later. It is clear that $\phi(x)\uparrow \infty$ as $|x|\uparrow \infty$. On the other hand, calculation shows
 \begin{align*}
  \fF[\phi] =& c a^2 + 2c\pare{-b x^2 + \frac{1-\gamma}{\gamma}  a (\lambda_0 + \lambda_1 x)' \Theta \rho x} + 2 c^2 a^2 \tilde{M} x^2\\
  &+ (1-\gamma)(r_0 + r_1 x) + \frac{1-\gamma}{2\gamma}(\lambda_0 + \lambda_1 x)' \Theta (\lambda_0 + \lambda_1 x)\\
  =& \pare{-2c b+ 2c\frac{1-\gamma}{\gamma}a \lambda_1' \Theta \rho + 2c^2 a^2 \tilde{M} + \frac{1-\gamma}{2\gamma} \lambda_1' \Theta \lambda_1}x^2 + \text{ lower order terms}.
 \end{align*}
 When $-b+ \frac{1-\gamma}{\gamma} a \lambda_1'\Theta \rho <0$, since $\gamma>1$, $\frac{1-\gamma}{2\gamma} \lambda_1' \Theta \lambda_1\leq 0$, we can choose sufficiently small $c$ such that $\fF[\phi] \downarrow -\infty$ as $|x|\uparrow \infty$. When $\lambda_1' \Theta \lambda_1>0$, then $\frac{1-\gamma}{2\gamma} \lambda'_1 \Theta \lambda_1<0$, we can also choose sufficiently small $c$ such that $\fF[\phi]$ has the same asymptotic behavior.  In both cases, $\fF[\phi]$ is bounded from above on $\Real$, hence Assumption \ref{ass: phi} is verified.

 \vspace{2mm}

 \noindent{\underline{Assumption \ref{ass: Q^0}}:} Consider the martingale problem associated to $\mathcal{L}^0 := [-bx - a (\lambda_0 + \lambda_1 x)' \rho] \partial_x + \frac12 a^2 \partial^2_{x}$ on $\Real$. Since all coefficients have at most linear growth, this martingale problem is well-posed and its solution, denoted by $\qprob^\rho$, satisfies $\frac{d\qprob^\rho}{d\prob} = \mathcal{E}\pare{\int-(\lambda_0 + \lambda_1 X_s)' \rho dW_s}_T$. Define $\qprob^0$ via
 \[
  \frac{d\qprob^0}{d\qprob} = \mathcal{E}\pare{-\int(\lambda_0+\lambda_1 X_s)'(\rho dW_s + \rho^\bot dW^\bot_s)}_T = \mathcal{E}\pare{\int-(\lambda_0 + \lambda_1 X_s)' dW^\rho_s}_T.
 \]
Argument similar to \eqref{eq: mart} implies that $\qprob^0$ is well defined. Therefore $\lambda$ in Assumption \ref{ass: Q^0} can be chosen as $\lambda_0 + \lambda_1 X$.

 To verify \eqref{eq: Q^0 int}, note
 \begin{equation}\label{eq: exp^psi-ou}
 \begin{split}
   \mathcal{E}\pare{\int (\lambda_0 + \lambda_1 X_s)' dW^0_s }^\psi_T
    &= \exp\pare{\frac12 (\psi^2-\psi) \int_0^T  |\lambda_0 + \lambda_1 X_s|^2 ds} \mathcal{E}\pare{\int \psi(\lambda_0 + \lambda_1 X_s)' dW^0_s}_T,
 \end{split}
  \end{equation}
  where  $W^0 := W^\rho + \int_0^\cdot (\lambda_0 + \lambda_1 X_s)ds$ is a $\qprob^0-$Brownian motion.  Following the construction of $\qprob^0$, similar argument shows that $\mathcal{E}\pare{\int \psi(\lambda_0 +\lambda_1 X_s)' dW^0_s}$ is a $\qprob^0-$martingale. Hence $\qprob^\psi$ can be defined via
  \[
   \frac{d\qprob^\psi}{d\qprob^0} := \mathcal{E}\pare{\int \psi(\lambda_0+\lambda_1 X_s)' dW^0_s}_T.
  \]
  Combining the previous two change of measures, the dynamics of $X$ can be rewritten as
  \[
   dX_t = \bra{(\psi-1) a \lambda_0' \rho - \pare{b - (\psi-1) a \lambda_1' \rho}X_t} dt + a dW^\psi_t,
  \]
  where $W^\psi := W + \int_0^\cdot(1-\psi) (\lambda_0 + \lambda_1 X_s)'\rho ds$ is a $1-$dimensional $\qprob^\psi-$Brownian motion. On the other hand, calculation shows, for any $\epsilon>0$,
  \begin{equation}\label{eq: exp Q0 ub}
  \begin{split}
   & \expec^{\qprob^0}\bra{e^{(\psi-1) \int_0^T r_+(X_s) ds} \mathcal{E}\pare{\int \lambda'(X_s) dW^{0}_s}^\psi_T} \\
   &= C \expec^{\qprob^{\psi}}\bra{\exp\pare{(\psi-1) \int_0^T (r_1 X_s)_+ ds + (\psi^2 - \psi) \lambda_0' \lambda_1 \int_0^T X_s ds + \frac12 (\psi^2 -\psi)\lambda'_1 \lambda_1 \int_0^T X_s^2 ds}}\\
   &\leq C_\epsilon  \expec^{\qprob^{\psi}}\bra{\exp\pare{\pare{\frac12 (\psi^2 -\psi)\lambda'_1 \lambda_1 +\epsilon}\int_0^T X_s^2 ds}},
  \end{split}
  \end{equation}
  where $C$ is a constant and $C_\epsilon$ is a constant depending on $\epsilon$.

  In order to appeal Lemma \ref{lem: int CIR} to calculate the expectation on the right hand side of \eqref{eq: exp Q0 ub}, let us introduce another measure $\tilde{\qprob}^\psi$ via $\frac{d\tilde{\qprob}^\psi}{d\qprob^\psi} = \mathcal{E}\pare{- (\psi-1) \lambda'_0 \rho W^\psi_T}$. Under this measure, $X$ has dynamics
  \[
   dX_t = - \pare{b- (\psi-1) a \lambda'_1 \rho} X_t dt + a d\tilde{W}^\psi_t,
  \]
  where $\tilde{W}^\psi:=W^\psi + \int_0^\cdot (\psi-1) \lambda'_0 \rho \,ds$ is a $\tilde{\qprob}^\psi-$Brownian motion.  Let $Y:= X^2$. It then has dynamics
 \[
  dY_t = \bra{a^2 - 2\pare{b-(\psi-1) a \lambda_1' \rho} Y_t} dt + 2a\sqrt{Y_t} d\tilde{W}^\psi,
 \]
 which is of the same type of $X$ in Lemma \ref{lem: int CIR}.

 Come back to \eqref{eq: exp Q0 ub}, H\"{o}lder's inequality implies, for any $\delta>0$,
 \begin{align*}
  &\expec^{\qprob^{\psi}}\bra{\exp\pare{\pare{\frac12 (\psi^2 -\psi)\lambda'_1 \lambda_1 +\epsilon}\int_0^T X_s^2 ds}} \\
  &= \expec^{\tilde{\qprob}^\psi}\bra{\frac{d\qprob^\psi}{d\tilde{\qprob}^\psi} \exp\pare{\pare{\frac12 (\psi^2 -\psi)\lambda'_1 \lambda_1 +\epsilon}\int_0^T X_s^2 ds}}\\
  &\leq \expec^{\tilde{\qprob}^\psi}\bra{\pare{\frac{d\qprob^\psi}{d\tilde{\qprob}^\psi}}^{\frac{1+\delta}{\delta}}}^{\frac{\delta}{1+\delta}}
  \expec^{\tilde{\qprob}^\psi}\bra{\exp\pare{(1+\delta)\pare{\frac12 (\psi^2 -\psi)\lambda'_1 \lambda_1 +\epsilon}\int_0^T X_s^2 ds}}^{\frac{1}{1+\delta}}.
 \end{align*}
 Observe that the first expectation on the right hand side is finite, since $\frac{d\qprob^\psi}{d\tilde{\qprob}^\psi} = \mathcal{E}\pare{(\psi-1) \lambda'_0 \rho \tilde{W}^\psi_T}$ has all finite moments. For the second expectation, we can choose sufficiently small $\delta$ and $\epsilon$ such that, according to Lemma \ref{lem: int CIR}, when
 \begin{equation}\label{eq: para OU ineq}
  \frac12 (\psi^2 -\psi) \lambda'_1\lambda_1 < \frac{4\pare{b- (\psi-1)a \lambda'_1 \rho}^2}{8a^2},
 \end{equation}
 the second expectation is finite. Now combining the previous estimates and \eqref{eq: exp Q0 ub}, we confirm \eqref{eq: Q^0 int}. Finally, note that \eqref{eq: para OU ineq} is exactly the assumption in Proposition \ref{prop: OU} ii).
\end{proof}

\bibliographystyle{siam}
\bibliography{biblio}

\end{document}